\documentclass[10pt,journal,compsoc]{IEEEtran}
\ifCLASSOPTIONcompsoc
\usepackage[nocompress]{cite}

\else
\usepackage{cite}
\fi
\ifCLASSINFOpdf
\usepackage[pdftex]{graphicx}
\else
\usepackage[dvips]{graphicx}
\fi
\usepackage{url}
\usepackage{booktabs} 
\usepackage{comment}
\usepackage[english]{babel}
\usepackage[utf8x]{inputenc}
\usepackage[T1]{fontenc}
\usepackage{bm}
\usepackage{amsmath}
\usepackage{amsfonts}
\usepackage{amssymb}
\usepackage{comment}
\usepackage{graphicx}
\usepackage{multicol}
\usepackage{nomencl}
\usepackage{subcaption}
\usepackage{mathtools}

\newcommand{\E}{\mathbb{E}}
\usepackage{bbm}

\usepackage{amsthm}
\newtheorem{theorem}{Theorem}
\newtheorem*{theorem*}{Theorem}
\newtheorem{lemma}{Lemma}

\newtheorem{fact}{Fact}
\newtheorem*{lemma*}{Lemma}
\newtheorem*{remark*}{Remark}
\newtheorem*{fact*}{Fact}
\newtheorem*{proof*}{Proof}

\newtheorem{proposition}{Proposition}
\newtheorem{definition}{Definition}
\newtheorem{corollary}{Corollary}

\DeclareMathOperator*{\argmin}{arg\,min}
\DeclareMathOperator*{\argmax}{arg\,max}
\usepackage{tikz} 
\usetikzlibrary{positioning}
\usetikzlibrary{shapes,arrows} 
	\tikzstyle{decision} = [diamond, draw, fill=white!20,
	text width=4.5em, text badly centered, node distance=2cm, inner sep=0pt]
	\tikzstyle{block} = [rectangle, draw, fill=white!20,
	text width=15em, text centered, rounded corners, node distance=3cm, minimum height=2.5em]
	\tikzstyle{smallblock} = [rectangle, draw, fill=white!20,
	text width=5em, text centered, rounded corners, node distance=3cm, minimum height=2.5em]
	\tikzstyle{line} = [draw, -latex']
	\tikzstyle{cloud} = [draw, ellipse,fill=white!20,  node distance=2cm,
	minimum height=2em]

\begin{document}
\title{A Capacity-Price Game for Uncertain Renewables Resources }

\author{Pan~Li,
	Shreyas~Sekar,~\IEEEmembership{Member,~IEEE,}
	and~Baosen~Zhang,~\IEEEmembership{Member,~IEEE}
	\IEEEcompsocitemizethanks{\IEEEcompsocthanksitem Pan Li is with Facebook Inc., Menlo Park, CA. This work was done when she was with the University of Washington. 
		E-mail: lipan.uw@gmail.com.
		\IEEEcompsocthanksitem Shreyas Sekar is with Harvard University, Boston, MA. This work was done when he was with the University of Washington.
		.
		E-mail: ssekar@hbs.edu.
		\IEEEcompsocthanksitem Baosen Zhang is with the Department of Electrical and Computer Engineering, University of Washington, Seattle, 98195.
		.
		E-mail: zhangbao@uw.edu.
		\IEEEcompsocthanksitem This work is partially supported by NSF grant CNS-1544160.
		}%
		 <-this 
	}

%


\IEEEtitleabstractindextext{
\begin{abstract}
Renewable resources are starting to constitute a growing portion of the total generation mix of the power system. A key difference between renewables and traditional generators is that many renewable resources are managed by individuals, especially in the distribution system. In this paper, we study the capacity investment and pricing problem, where multiple renewable producers compete in a decentralized market. It is known that most deterministic capacity games tend to result in very inefficient equilibria, even when there are a large number of similar players. In contrast, we show that due to the inherent randomness of renewable resources, the equilibria in our capacity game becomes efficient as the number of players grows and coincides with the centralized decision from the social planner's problem. This result provides a new perspective on how to look at the positive influence of randomness in a game framework as well as its contribution to resource planning, scheduling, and bidding. We validate our results by simulation studies using real world data.
\end{abstract}
}
%

\maketitle

\section{Introduction}\label{sec:intro}


Distributed renewable resources are starting to play an increasingly important role in energy generation. For example, the installation of photovoltaic (PV) panels across the world has grown exponentially during the past decade~\cite{TimilsinaEtAl2012}. These renewable resources tend to be different from traditional large-scale generators as they are often spatially distributed, leading to many small generation sites across the system. The proliferation of these individual renewable generators (especially PV) has allowed for a much more flexible system, but also led to operational complexities because they are often not coordinated~\cite{Winter1991}.  Recently, there has been strong regulatory and academic push to allow these individual generators to participate in a market, hoping to achieve a more efficient and streamlined management structure~\cite{Heinrichs2013,Kalathil2017}. Therefore, in this paper we study an investment game where individual firms decide their installation capacities of PV panels and compete to serve the load.


Currently, there are several lines of fruitful research on the investment of solar energy resources. The common challenge in these works is to address the \emph{pricing} of solar energy, since once installed, power can be produced at near zero operational cost. In \cite{CoutureEtAl2010}, feed-in tariffs (fixed prices) are used to guide the investment decisions. In \cite{WustenhagenEtAl2013,FletenEtAl2007}, risks about future uncertainties in prices are taken into account, although these prices are assumed to be independent of the investment decisions. Instead of exogenously determined feed-in tariffs,   \cite{MenanteauEtAl2003,Hirth2013,NegashEtAl2015} study incentive based pricing, arguing that the price of solar energies should match their market value, which is the revenue that those resources can earn in markets, without the income from subsidies. However, the investment question of how to decide the capacity of each solar installation is not considered.


A common assumption made in existing studies is that an operator (or utility) makes centralized decisions about the capacity of the solar installation at different sites~\cite{GrossmannEtAl2013,Aragones-BeltranEtAl2010} as well as the market clearing price upon aggregating the bids from individual producers. On the other hand, since small PV generation is mostly privately owned, it is arguably more realistic to consider an environment where both the investment and the pricing decisions are made in a decentralized manner.
In this market setting, the PV owners compete by making their own investment and price bidding decisions, based on the information they have, as opposed to a centralized decision made by a single operator. In this paper, we are interested in understanding these strategic decisions, particularly in the electrical distribution system.

The competition between individual producers in a decentralized market for electricity is normally studied either via the Cournot model or the Bertrand model~\cite{Kirschen2004}. In the former, the producers compete via quantity, while in the latter they compete via price. In this work, we adopt the Bertrand competition to model price bidding since it is a more natural process in the distribution system, where there is no natural inverse demand function (required by the Cournot competition model)~\cite{BorensteinEtAl2003,Crampes06capacitycompetition}. Then the investment game becomes a \emph{two-level} game as shown in Fig.~\ref{fig:decentralized}. For any given capacities, the producers compete through the Bertrand model to determine their prices to satisfy the demand in the system. Then the outcome of this game feeds into an upper level capacity game, where each producer determines its investment capacities to maximize its \emph{expected profit}. Fig.~\ref{fig:setup1} depicts a detailed comparison between this two-level Bertrand game and a more centralized marketplace where the utility selects the clearing price.


\begin{figure*}[htbp]
	\centering
	\begin{minipage}[t]{0.465\linewidth}
		\centering
			\begin{tikzpicture}[node distance = 2cm, auto]
		\node [block, node distance = 2cm] (capacity) {Energy producers submit a per-unit bid and a supply cap};
		\node [block, below of=capacity, node distance = 2cm] (price) {The designer selects $k$ buyers having the lowest bids so that the total supply matches the demand};
		\node [block, below of=price, node distance = 2.5cm] (pay) {The market clearing price equals the $k+1^{th}$ lowest bid. Each of the $k$ buyers receives a per-unit payment equal to the clearing price. };
		\path [line] (capacity) -- (price);
		\path [line] (price) -- (pay);
		\end{tikzpicture}
		\subcaption{A description of a centralized mechanism where a number of producers submit bids and a central authority sets the market clearing price so that the overall supply coincides with demand.}
		\label{fig:centralized}
	\end{minipage}~~
	\begin{minipage}[t]{0.465\linewidth}

	\begin{tikzpicture}[node distance = 2cm, auto]
\node [block, node distance = 2cm] (capacity) {Energy producers determine investment capacity to maximize expected profit.};
\node [block, below of=capacity, node distance = 2cm] (price) {Given capacity, each producer bids according to a Bertrand model to maximize revenue.};
\node [block, below of=price, node distance = 2cm] (buyer) {Buyers purchase sequentially starting with the lowest bid until total demand is met};
\path [line] (capacity) -- (price);
\path [line] (price) -- (buyer);

\draw [dashed,->] (price) --++ (4.7cm,0cm) |- node[above]  {optimal revenue} (capacity);
\end{tikzpicture}
		\subcaption{An illustration of the two level game between a group of solar energy producers. In the higher level, the producers determine their investment capacity to maximize expected profit, and the lower game determines the pricing of solar energy through Betrand competition.} 
		\label{fig:decentralized}
	\end{minipage}
	\caption{Centralized vs. Decentralized market mechanisms.}
	\label{fig:setup1}
\end{figure*}


This type of two-level game was studied in~\cite{AcemogluEtAl09} in the context of communication network expansion. They showed that Nash equilibria exist, but the efficiency of any of these equilibria are bad compared to the social planner's (or operator's) solution. More precisely, as the number of players grows, the social cost of all of the equilibria grow with respect to the cost of the social planner's problem. Therefore, instead of increasing efficiencies, competition can be arbitrarily bad. A similar intuition has existed in traditional power system investment problems, where the market power of the generators is highly regulated and closely monitored~\cite{Wu2006}.

In this paper, we show that contrary to the result in \cite{AcemogluEtAl09}, \emph{the investment game between renewable producers leads to efficient outcomes under mild assumptions.} More precisely, 1) the investment capacity decisions made by the individual producers match the capacity decisions that would be made by a social planner; 2) as the number of producers increases, the equilibria of the price game approaches a price level that allows the producers to just recover their investment costs. The key difference comes from the fact that renewables are \emph{inherently random}. Therefore instead of trying to exploit the ``corner cases'' in a deterministic setting as in \cite{AcemogluEtAl09,Wu2006}, the uncertainties in renewable production naturally induces conservatism into the behavior of the producers, leading to a drastic improvement of the Nash equilibria in terms of efficiency. Therefore, uncertainty helps rather than hinders the efficiency of the system.

To analyze the equilibria of the game, our work builds on the results in~\cite{TaylorEtAl16}. In~\cite{TaylorEtAl16}, the authors discuss the price bidding strategies in markets with exactly two renewable energy producers. They show that a unique mixed pricing strategy always exists given that the capacity of those producers are fixed beforehand. They extend it to a storage competition problem in later work~\cite{TaylorEtAl17}. However, this work did not address the strategic nature of the capacity investment decisions, nor did it consider markets with more than two producers.

In our setting, we explicitly consider the joint competition for capacity considering each player's investment cost, as well as the bidding strategy to sell generated energy. This problem is neither studied in traditional capacity investment games (randomness is not considered)~\cite{Reynolds1987,SmitEtAl2012}\footnote{The work in \cite{GoyalEtAl2007,FrutosEtAl2011} studies an investment game where the demand curve is uncertain, but under a very different context than ours} nor in competition of renewable resources (investment strategy is considered)~\cite{MenanteauEtAl2003,LangnissEtAl2003}.
To characterize the Nash equilibria in the two level capacity-pricing game, we consider two performance metrics. The first is social cost, which is the total cost of a Nash equilibrium solution with respect to the social planner's objective. The second is market efficiency, which measures the market power of the energy producers. As a comparison, the results in \cite{AcemogluEtAl09} show that in a deterministic capacity-pricing game, as the number of producers grows, \emph{neither the social cost nor the efficiency improves at equilibrium}. In contrast, we show that a little bit of randomness leads to improvements on both metrics. Specifically, we make the following two contributions:
\begin{enumerate}
	\item We consider a two level capacity-pricing game between multiple renewable energy producers with random production. We show that contrary to commonly held belief, randomness improves the quality of the Nash equilibria.\footnote{This is conceptually similar to the results obtained in~\cite{ZhangEtAl2015}, where randomness increases the efficiency of Cournot competition.}
	\item We explicitly characterize the Nash equilibria and show that the social cost and efficiency improve as the number of producers grows.
\end{enumerate}

The rest of the paper is organized as follows. Section \ref{sec:model} motivates the problem set up and details the modeling of both the decentralized and centralized market. Section \ref{sec:eval} formally introduces the evaluation metrics for our setting. Section \ref{sec:mainresults} presents the main results of this paper, i.e., the relationship between the proposed decentralized market and the social planner's problem, and the analysis on the efficiency of the game in the decentralized market. Proofs for the main theorems are left in the appendices for interested readers. The simulation results are shown in Section \ref{sec:sim} followed by the conclusion in Section \ref{sec:conclusion}.

\section{Technical preliminaries}\label{sec:model}
\subsection{Motivation}
Traditionally, power systems are often built and operated in a centralized fashion. The system operator acts as the social planner by aggregating the producers and makes centralized decisions on investment and scheduling~(as shown in Fig. \ref{fig:social}). The goal of the social planner is to  maximize the overall welfare of the whole system--- this includes optimizing the costs incurred due to the investment and installation, and the cost paid by the consumers.

\begin{figure}[htbp]
	\begin{minipage}[t]{0.45\linewidth}
		\includegraphics[width=0.8\linewidth]{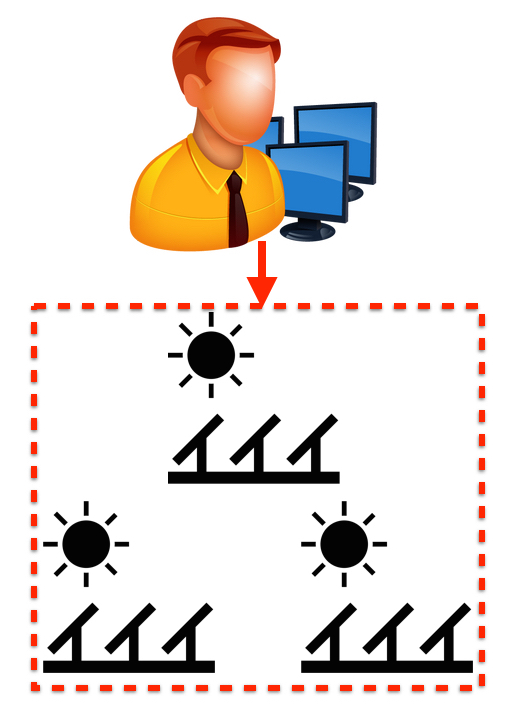}
		\subcaption{Centralized.}
		\label{fig:social}
	\end{minipage}%
	\hfill%
	\begin{minipage}[t]{0.45\linewidth}
		\includegraphics[width=0.8\linewidth]{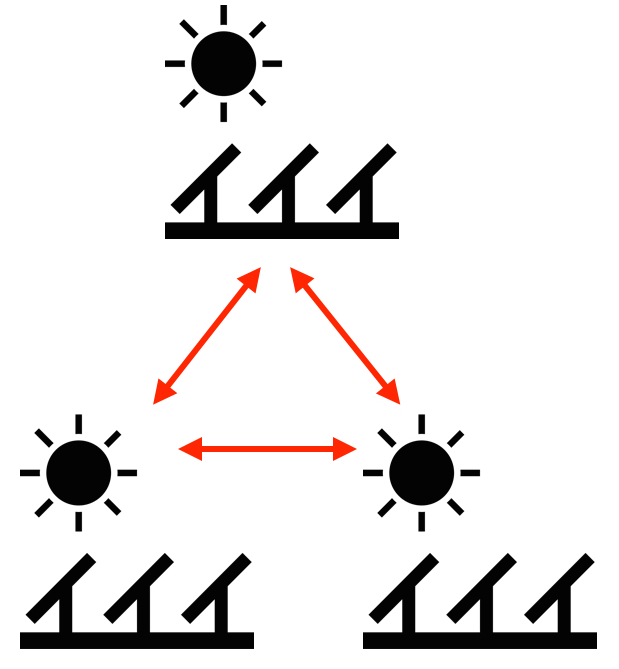}
		\subcaption{Decentralized.}
		\label{fig:game}
	\end{minipage}
	\caption{Centralized vs. decentralized market setup. }
	\label{fig:setup}
\end{figure}

However, as distributed energy resources~(DERs) start to disperse across the power distribution network, the centralized setup becomes difficult to maintain and manage. DERs such as rooftop PV cells are small, numerous, and owned by individuals, allowing them to act as producers and choose their own capacities and prices. Consequently, managing these resources through a decentralized market~(as shown in Fig.~\ref{fig:game}) is starting to gain significant traction in the power distribution system.

Several issues arise in a decentralized market. Chief among them is that
it is not clear whether the decentralized market achieves the same decision as if there were a central planner maximizing social welfare. The competition between energy producers is suboptimal if the following occurs:
\begin{itemize}
	\item If the investment decisions by the competing producers deviate from the social planner's decision: this means that the competition is sub-optimal when it comes to finding a socially desirable investment plan.
	\item If the bidding strategy leads to a higher payment from electricity consumers than that from the social planner's decision, it means that the energy resource producers are taking advantage of the buyers and the market is not efficient.
\end{itemize}
Both of these adverse phenomena can happen in decentralized markets in the absence of uncertainty~\cite{AcemogluEtAl09,JohariEtAl2010,OzdaglarEtAl2011}, even when there are a large number of individual players. However, the rest of this paper shows that neither of them occur in a decentralized market with renewables resources having random generation. We show that the inherent \emph{uncertainty} in the production naturally improves the quality of competition. We start by formally introducing the game in the next section.


\subsection{Renewable Production Model}
Throughout this paper, we denote some important terms by the following:
\begin{itemize}
	\item $Z_i$: Random variable representing the output of producer $i$, scaled between 0 and 1. The moments of $Z_i$ are denoted as $\E Z_i = \mu_i$, $\E (Z_i - \E Z_i)^2 = \sigma_i^2$ and $\E |Z_i - \E Z_i|^3 = \rho_i$.
	\item $C_i$: Capacity of producer $i$.
	\item $D$: Total electricity demand in the market
	\item $N$: Number of producers in the market.
	\item $\gamma_i$: Investment cost for unit capacity for producer $i$.
	\item $\xi$: Efficiency of the game equilibrium.
	\item $\bm{x}_{-i}$: The quantities chosen by all other producers except $i$, that is $\bm{x}_{-i} = [x_1, \dots, x_{i-1}, x_{i+1}, \dots, x_N]$.
	\item $(x)^+ \coloneqq \max(x,0)$.
	\item $(x)^- \coloneqq \min(x,0)$.
\end{itemize}
\textbf{Renewable Production Model:} When producer $i$ invests in a capacity of $C_i$, its actual generation of energy is a random variable given by $C_i Z_i \leq C_i$. That is, due to the randomness associated with renewables, its realized production may not equal its maximum capacity.

We make the following assumption on the $Z_i$'s: \\
\noindent \textbf{[A1]} We assume that the random variable $Z_i$ has support $[0,1]$ and its density function is bounded and continuous on its domain. This assumption is mainly made for analytical convenience and captures a wide range of probabilistic distributions used in practice, e.g., truncated normal distribution and uniform distribution. Furthermore, we assume that $Z_i$ is not a constant, so $\E (Z_i - \E Z_i)^2 = \sigma_i^2 > 0$.

{
We also make the following assumptions on demand $D$:\\
\noindent \textbf{[B1]} We assume that the random variable $D$ is non-negative and that $D$ is bounded, i.e., $ D_{min} \leq D \leq D_{max}$ where $\frac{D_{min}}{D_{max}}$ is a constant bounded away from zero. }

\subsection{Competition in decentralized markets}\label{subsec_cap_game_intro}

Consider $N$ renewable producers who compete in a decentralized market. Each producer needs to decide two quantities: capacity (sizing) and the corresponding everyday price bidding strategy. To make this decision, each producer needs to take into consideration the fact that larger capacities lead to higher investment costs but may also result in enhanced revenue due to increased sales.  If the invested capacity is low, then the investment cost is low but the producer risks staying out of the market because of less capability to provide energy. Therefore competition requires non-trivial decision making by the decentralized stakeholders. In this paper, we consider both cases where each producer has either same or different investment cost. To start with, let us assume $\gamma_i=\gamma$ for all $i$. This assumption is true to the first order since the solar installation cost in an area is roughly the same for all the consumers. A more general case with different $\gamma_i$'s is left to Section \ref{sec:asym_gamma}. Since the producers need to compete for capacity based on revenue (which is determined by optimal bidding), we refer to the capacity competition---how much to invest---as the \emph{capacity game} and the pricing competition, i.e., how much to bid, as the \emph{price sub-game}. {To generalize the scenario we assume that the price bidding is a one-time action that approximates an amortized version of a series of continuous bidding actions.}
\subsection{Capacity game}\label{subsec:capacity_game}
The ultimate decision for the producers is to determine the optimal capacity to invest in. Suppose that the capacity is denoted by $C_i$ for each producer $i$, then each producer's objective is to maximize its profit, which is specified as:
\begin{equation}\label{eq:cap_game}
\text{(Profit)} \quad \pi_i(\bm{C}_i, C_{-i}) - \gamma C_i, \forall i,
\end{equation}
where $\pi_i(\bm{C}_i, C_{-i})$ is the payment (revenue) from consumers to producer $i$ when its capacity is fixed at $C_i$ and the others' capacities are fixed at $C_{-i}$.  This payment is determined by the price sub-game given that a capacity decision is already made, i.e., $\bm{C} = [C_1, C_2, ..., C_N]$: we leave a detailed discussion of the revenue and the price sub-game to Section \ref{subsec:price_game}. The term $\gamma C_i$ represents the investment cost.  


Since we are in a game-theoretic scenario, the appropriate solution concept is that of a Nash equilibrium. Specifically, a capacity vector $\bm{C}^{\diamond} = [C^{\diamond}_1, C^{\diamond}_2, ..., C^{\diamond}_N]$ is said to be a Nash equilibrium if:
\begin{equation}\label{eq:cap_game_nash}
C^{\diamond}_{i} = \argmax_{C_i \geq 0 } \pi_i(\bm{C}_i, C^{\diamond}_{-i}) - \gamma C_i, \forall i,
\end{equation}

The Nash equilibrium shown in \eqref{eq:cap_game_nash} is interpreted as the following: each producer $i$ chooses a capacity $C_i$ such that given the optimal capacity strategy of the others, there is no incentive for this producer to deviate from this capacity $C_i$. Note that while choosing its capacity, each producer implicitly assumes that its resulting revenue is decided by the solution obtained via the price sub-game. {Note that for each PV producer to make a decision according to \eqref{eq:cap_game_nash}, it needs to know the distribution of demand $D$, where random variable $D$ can be seen as discounted back from infinite horizon so we do not explicitly model $D_t$ in this manuscript. It also needs to know the distribution of solar radiation $Z_i$,  and investment price $\gamma$ as shown in next Section. Distribution of $D$ can be obtained through historical data, $Z_i$ can be estimated based on geographical differences, and investment price should be publicly available given that PV producers use similar materials to fabricate panels. Therefore each PV player can independently play its strategy without accessing private information of others.}


\subsection{Price sub-game}\label{subsec:price_game}
In this section, we explicitly characterize the payment function  $\pi_i(\bm{C}_i, C_{-i})$ at the equilibrium solution of the price sub-game for a fixed capacity vector $(C_i)_{1 \leq i \leq N }$. The producers now compete to sell energy at some price $p_i$. This is known as the Bertrand price competition model, where the consumer prefers to buy energy at low prices. In this model, consumers resort to buying at a higher price only when the capacity of all the lower-priced producers are exhausted. Suppose that the profit for producer $i$ when the producers bid at $\bm{p} = [p_1, p_2, ..., p_N] = [p_i, \bm{p}_{-i}]$ is denoted by $\pi_i(C_i, \bm{C}_{-i}, p_i, \bm{p}_{-i})$. We make the following assumption about the prices:\\
\noindent \textbf{[A2]} The customers have the options to buy energy at unit price from the main grid. \\
This assumption follows the current structure of a distribution system, where customers have access to the main grid at a fixed price, and here we normalize the price to $1$. Equivalently, this can be thought as the value of the lost load in a microgrid without a connection to the bulk electric system~\cite{Katiraei2006}.


As shown in~\cite{AcemogluEtAl09,TaylorEtAl16}, there is no pure Nash equilibrium on $p$ for the price sub-game. Intuitively, this means that no player can bid at a single deterministic price and achieve the most revenue, since the other players can undercut by a tiny amount and sell all their generation. Therefore no player settles on a pure strategy. Such a situation particularly arises where each producer is small ($C_i \leq D, \forall i$), but the aggregate is large ($\sum_i C_i > D$, where $D$ denotes the total demand in the market).




However, there exists a \emph{mixed Nash equilibrium on price $p$}, where the optimal bids follow a distribution such that the bids of each DER are independent of the rest. Informally, this implies that each producer $i$ draws its price $p_i$ from a distribution $\mathcal{P}^*_i$, which maximizes its expected revenue given the distributions of the other producers. For example, the price distribution of a two player Bertrand model is given in \cite{TaylorEtAl16}. For our purpose, the exact form of the optimal price distribution is not of particular interest. The quantity of interest is the form of the revenue function	, i.e., the expected payment, resulting from this random price bidding. Let us denote the expected payment to producer $i$ based on the optimal random price by $\pi_i(C_i, \bm{C}_{-i}) = \E_{\bm{p} \sim \mathcal{P}^*_1 \times \ldots \times \mathcal{P}^*_N } \pi_i(C_i, \bm{C}_{-i},  p_i, \bm{p}_{-i})$. Proposition \ref{prop:expected_profit} characterizes the optimal payment to each producer:
\begin{proposition}\label{prop:expected_profit}
	Given any solution $(C_1, C_2, \ldots, C_N)$ having $C_1 \leq C_2 \leq \ldots \leq C_N$, the expected payment received by producer $i$ in the equilibrium of the pricing sub-game is given by:
	\begin{align}
	\pi_i(C_1, \ldots, C_N) & = \pi_N(C_1, \ldots, C_N) \frac{C_i \E[\min(Z_i, \frac{D}{C_i})] }{C_N \E[\min(Z_N, \frac{D}{C_N})]}\label{eqn_paymentasymmetric} .
	\end{align}
	{Moreover, the expected payment received by producer $N$ with the largest capacity investment is given by:
	\begin{equation}\label{eq:eq_pi}
\pi_N (C_N, \bm{C}_{-N}) =\E_D {\E}_{Z_N, Z_{-N}} \min\{(D-\sum_{j \neq N} Z_{j}C_{j})^{+}, Z_NC_N \}.
\end{equation}}
\end{proposition}

	{As a consequence of Proposition~\ref{prop:expected_profit}, we have that if the capacities are symmetric at equilibrium, i.e., $C_1 = C_2 = \cdot = C_N$, and $Z_i$'s are identically distributed for all $1 \leq i \leq N$, the expected payment of producer $i$ is given by:
	$$	\pi_i (C_i, \bm{C}_{-i}) =\E_D {\E}_{Z_i, Z_{-i}} \min\{(D-\sum_{j \neq i} Z_{j}C)^{+}, Z_iC_i \}.	$$}

A complete proof of Proposition \ref{prop:expected_profit} is deferred to the Appendix. Let us now understand Proposition \ref{prop:expected_profit} for the symmetric investment solution. Equation~\eqref{eq:eq_pi} denotes the payment received by producer $i$ when it bids deterministically at price $p_i = 1$ and all of the other producers bid according to their mixed pricing strategy. By assumption A2, this player bids at the highest possible price. Then the amount of energy sold equals the minimum of the leftover-demand from the market {($\E (D-\sum_{j \neq i} Z_{j}C_{j})^{+})$} and the player's actual production ($C_iZ_i$). Since $p_i = 1$ belongs to the support of the mixed pricing strategy adopted by this player, one can use well known properties of mixed Nash equilibrium~\cite{AcemogluEtAl09, TaylorEtAl16} to argue that producer $i$'s payment at this price equals the expected payment received by this producer at the equilibrium for the pricing sub-game.

\section{Evaluation metrics}\label{sec:eval}

\subsection{Social planner's problem}\label{subsec:social}
One essential characteristic of a game is its cost as compared to a centralized decision. In this section, we present the benchmark cost that we consider; in particular, we focus on the social cost minimization achieved by a social planner controlling the producers. In Section \ref{sec:eval_decentralized} we give more details on the definition of game efficiency as compared to this benchmark.

Suppose that these producers are managed by a social planner in a centralized manner. The purpose of the social planner is to fulfill demand while minimizing the total cost by deciding the investment capacities of the producers. {Since we approximate the continuous actions of energy buying and selling from PV producer to a one-time action as in Section \ref{subsec_cap_game_intro},} the social planner thus wants to minimize \emph{social cost} in the following form:
\begin{equation}\label{eq:social_opt}
\bm{C}^* = \argmin_{C_i \geq 0, \forall i \in \{1, 2, ..., N\}} \  \sum_{i=1}^N \gamma C_i +  \E \{(D-\sum_{i=1}^N Z_iC_i)^+ \},
\end{equation}
{where $\bm{C}^* = [C_1^*, C_2^*, ..., C_N^*]$ is the optimal capacity decision from the social planner for each producer $i$, and expectation is over all randomness, i.e., $Z_i$'s and $D$. In what follows, we adopt this routine if not specifcied otherwise.} {The social cost presented in \eqref{eq:social_opt} is composed of two terms. The first term is the total investment cost which is linear in the capacities, and the second term is the imbalance cost in buying energy from electricity grid if the renewables cannot satisfy the demand. These two terms represent the tradeoff between investing energy resources and buying energy from conventional generators in order to meet the electricity demand.}

\subsection{Performance of the decentralized market}
\label{sec:eval_decentralized}

Given the definition of the equilibrium solutions due to both price and capacity competition, a natural question is to evaluate the performance of the decentralized market: i.e., \emph{does competition result in efficiency?}. As mentioned previously, we measure this efficiency via two metrics: the social cost of the decentralized capacity investment compared to that achieved by the social planner, and the total investment cost compared to the payments made by the demand.

\noindent \textbf{Example.} Let us consider a one-player case with {deterministic demand $D$}, where there is only one producer participating in the electricity market. 
We further assume that the random output of this plant follows a uniform distribution, i.e., $Z_1 \sim \text{unif}(0, 1)$. Suppose that $\gamma < \frac{1}{2}$, otherwise there is no incentive to enter the market. The social planner's optimization is reduced to :
\begin{equation}
	{C}^*_1 = \argmin_{C_1} \  \gamma C_1 + \E (D-Z_1C_1)^+,
\end{equation} where $C_1^* = \sqrt{\frac{D^2}{2\gamma}}$. In this case, the total investment cost is $D\sqrt{\frac{\gamma}{2}}$: in a centralized scenario, one can imagine that this is the price charged by the social planner to the demand, and thus there is no `markup'.

Let us now take a look at the decentralized market. Since there is only one producer, the decentralized investment strategy clearly coincides with that of the social planner. The payment from the demand to the producer as per~\eqref{eq:eq_pi} is $\E \min\{D, Z_1 C_1^{\diamond}\}  = D (1 - \sqrt{\frac{\gamma}{2}}) > D\sqrt{\frac{\gamma}{2}}$ (when $\gamma <\frac{1}{2}$). This suggests that the producer is exploiting its market power to considerably improve its profit and the benefits of renewables are not being transferred to the consumers.


\noindent \textbf{Market Efficiency}
As noticed in the above example, inefficiency arises due to the high prices felt by the demand in the decentralized market. Formally, we define market efficiency as the ratio between the investment cost paid by the producers to the total payment received by the producers at any equilibrium of the capacity price game. Therefore, efficiency takes the following form:
\begin{equation}\label{eq:eff}
	\xi = \frac{\gamma \sum_{i = 1}^N C^{\diamond}_i }{\sum_{i = 1}^N\pi(C^{\diamond}_i, \bm{C}^{\diamond}_{-i})}.
\end{equation}


A ``healthier'' game should achieve a higher $\xi$ that is as close as to 1. 
This means that the producers should bid at the prices that cover their investment cost, so that bidding is efficient and does not take advantage of the electricity consumers. A particularly interesting question is whether competition leads to increased efficiency as the number of producers in the market increases. We formalize this notion below.

\begin{definition}\label{def:eff}
	We define the efficiency of a Nash equilibrium in a capacity game illustrated in \eqref{eq:cap_game} by $\xi$.
	The capacity game is asymptotically efficient when $\xi \rightarrow 1$ as $N \rightarrow \infty$ for every Nash equilibrium.
\end{definition}


{Now the question of interest is 1) whether uncertainty in generation deteriorates or improves the market efficiency of the game, and 2) whether efficiency increases as the number of players in the game increases. In the following sections, we will see that without randomness in the generation, the producers are able to charge a relatively high price for energy, which makes the game less efficient. Interestingly, when producer's generation becomes uncertain, the game becomes more efficient as more producers are involved in the decentralized market.}

\emph{Inefficiency due to Social Cost:} When there are multiple producers, it is possible that even the investment decisions may not coincide with that of the social planner. Therefore, a second source of inefficiency is the social cost due to the capacity investment, as defined in~\eqref{eq:social_opt}. More concretely, we compare the social cost of the equilibrium solution $({C}^{\diamond}_i)$ with that of the social cost of the planner's optimal capacity $({C}^*_i)$--- clearly, the latter cost is smaller than or equal to the former.

\subsection{Deterministic game}
Before moving on to the main results, we highlight the (in)efficiency of the equilibrium in the \emph{deterministic} version of the capacity game, i.e., one without production uncertainty where $Z_i = 1$ with probability one. Understanding the inefficiency of this deterministic game is the starting point for us to better gauge the effects of uncertainty in investment games.

We begin with the social planner's problem with fixed $D$, which in the absence of uncertainty can be formulated as follows:
\begin{equation}
	\min_{C_i \geq 0, \forall i \in \{1, 2, ..., N\}} \gamma \sum_{i = 1}^N C_i + (D - \sum_{i = 1}^NC_i)^+.
\end{equation}

Every solution with non-negative capacities that satisfies\\$\sum_{i=1}^N C^*_i = D$ optimizes the above objective --- this includes the symmetric solution  $C^*_1 = C^*_2 = \dots = C^*_N = \frac{D}{N}$. Moving on to the decentralized game with deterministic energy generation, we can directly characterize the equilibrium solutions using the results from~\cite{AcemogluEtAl09}.  Specifically, by applying Proposition 13 in that paper, we get although there are multiple equilibrium solutions, every such solution $(C_i)_{1 \leq i \leq N}$ satisfies $(i)$ $\sum_{i=1}^N C_i = D$, and $(ii)$ $\pi_i(C_i, \bm{C}_{-i}) = C_i$.  The second result implies that at every equilibrium, each producer charges a price that is equal to the electricity price of one from the main grid. Finally, by applying~\eqref{eq:eff}, we can characterize the efficiency in terms of the investment cost $\gamma$:
%
%
%
%

\begin{equation}
	\label{eq:effdet}
	\xi = \frac{\gamma D}{ D} = \gamma.
\end{equation}

\emph{Why is this result undesirable?} First note that when $\gamma < 1$,~\eqref{eq:effdet} implies that the deterministic game is inefficient at \emph{every Nash equilibrium}. In fact, using the results from~\cite{AcemogluEtAl09}, one can deduce that the system is inefficient even when different producers have different investment costs. Perhaps more importantly, the costs of investment as well as the market price of renewable energy have dropped consistently over the past decade and are expected to continue doing so in the future~\cite{taylor2015renewable,barbose2017tracking,margolis2017q4}. In this context, Equation~\eqref{eq:effdet} has some stark implications, namely that \emph{as $\gamma$ (the investment price) drops in the long-run, the efficiency actually becomes worse ($\xi \to 0$ as $\gamma \to 0$)}, i.e., the improvements in renewable technologies do not benefit the electricity consumers.



\section{Main results}\label{sec:mainresults}
To better illustrate the results, In this section, we first assume that $\gamma_i$'s are the same across all producers and characterize the capacity decision from the social planner's problem. We then illustrate the relationship between the decentralized market, and the social planner's problem in the centralized market. We also give a thorough analysis on the efficiency of the decentralized market. We begin by considering the case where the capacity generated by the producers are independent of each other and then move on to the correlated case. In the end of this section, we extend the result to asymmetric $\gamma_i$'s. All of the proofs from this section can be found in the appendix.
\subsection{Social planner's optimal decision}
An immediate observation of the socially optimal capacity as described in \eqref{eq:social_opt} is that if the randomness is independent and identical across different producers, the socially optimal capacity is symmetric:
\begin{theorem}\label{th_symm2}
	{If the random variables $Z_i$ are i.i.d. and satisfy assumption A1 and $\gamma_i = \gamma, \forall i$, then the optimal capacity obtained by \eqref{eq:social_opt} is symmetric, i.e., $C_1^* = C_2^* = \cdots = C_N^* = C^*$.}
\end{theorem}

{Theorem \ref{th_symm2} states that when the investment cost per unit capacity is the same across all producers, and the random variable is i.i.d., then the optimal decision for the social planner is to treat all producers equally and invest the same amount of capacity for each producer. In reality, the randomness due to renewable sources can be correlated and Section \ref{subsec: correlation} shows that Theorem~\ref{th_symm2} stills holds under some conditions on the nature of the correlation.


	\subsection{Existence and Social Cost}\label{sec:comp}
	Now that we have captured the structure of the socially optimal capacity decision, we want to address the issue of whether or not the capacity price game admits Nash equilibrium solutions in the decentralized market. A second question concerns the social cost of Nash equilibria when compared to the optimum investment decision adopted by a social planner.  As discussed in Section~\ref{sec:eval_decentralized}, one of the two sources of inefficiency in decentralized stems from the fact that the social cost of equilibrium solutions may be larger than that of the central planner's solution. {We present Theorem~\ref{th:social_nash} which addresses both of these questions by proving the existence of {a Nash equilibrium that coincides with the socially optimal capacity decision. Characterizing other Nash equilibriums is left to Theorem \ref{maintheorem} in Section \ref{subsec:eff}.}}

	\begin{theorem}\label{th:social_nash}
		There is a Nash equilibrium that satisfies \eqref{eq:cap_game_nash},  which also minimizes the social cost. That is, $(C^*, C^*, \ldots, C^*)$ is a Nash equilibrium.
	\end{theorem}

	Therefore, existence is always guaranteed in our setting. More importantly, {Theorem \ref{th:social_nash} provides an interesting relationship between the centralized decision that minimizes social cost, and the decentralized decision where producers seek to maximize profit. It states that the game yields a socially optimal capacity investment solution as if there were a social planner controlling the producers.} {In addition, as we will show later in Section \ref{sec:uniqueEQ}, this Nash equilibrium is the unique symmetric equilibrium in the capacity game. For the following sections, we use $C^*$ to denote both the socially optimal capacity decision and this Nash equilibrium.}

	\subsection{Efficiency of Nash equilibrium}\label{subsec:eff}
	Although the capacity price game studied this work admits a Nash equilibrium that minimizes the social planner's objective, there may also exist other equilibria that result in sub-optimal capacity investments.
	How do these (potential) multiple equilibria look like from the consumers' perspective, i.e., is the price charged to consumers larger than the investment? In this section, we show a surprising result: the two-level capacity-pricing game is asymptotically efficient. That is,
	as $N \to \infty$, the total payment made to the producers approaches the investment costs for \emph{every Nash equilibrium}. The reason for this startling effect is that as the number of producers competing against each other in the market increases, with the presence of uncertainty, the market power of these producers decreases and the efficiency of the game equilibrium increases. We first present our main theorem with i.i.d. generation.


	\begin{theorem}
		\label{maintheorem}
		Let $({C}^{\diamond}_1, {C}^{\diamond}_2, \ldots, {C}^{\diamond}_N)$ denote any Nash equilibrium solution in an instance with $N$ producers and {$N > \frac{1}{\frac{D_{min}}{D_{max}}\gamma}$}, where $\gamma_i = \gamma, \forall i$. Then, as long as the $Z_i$'s are i.i.d and satisfy assumption A1, we have that:
		$$\sum_{i=1}^N \pi_i({C}^{\diamond}_1, {C}^{\diamond}_2, \ldots, {C}^{\diamond}_N) \leq \gamma \sum_{i=1}^N {C}^{\diamond}_i + \alpha N^{-c},$$

		where $\alpha, c > 0$ are constants that are independent of $N$. Therefore, as $N \to \infty$, $\xi \to 1$, where $\xi$ denotes the market efficiency due to \emph{any} Nash equilibrium solution.
	\end{theorem}

	%
	Combining Theorems~\ref{th:social_nash} and~\ref{maintheorem} yields that if we restrict the game to only have the symmetric equilibrium, then the equilibrium minimizes the social cost and the game is asymptotically efficient. Moving beyond the symmetric equilibrium, Theorem \ref{maintheorem} states that any Nash equilibrium obtained from the capacity game is efficient, that the collected payment (revenue) tends to exactly cover the investment cost. This further suggests that the capacity game described in \eqref{eq:cap_game} elicits the true incentive for the producers to generate energy.



	\subsection{Correlated generation}\label{subsec: correlation}
	In reality, renewable generation due to multiple entities in a power distribution network is usually correlated with each other because of geographical adjacencies. We assume that the randomness of each producer's generation can be captured as an additive model written as the following:
	\begin{equation}\label{eq:decomposition}
		Z_i = \hat{Z}_i + \bar{Z}.
	\end{equation}
	{The model in \eqref{eq:decomposition} captures the nature of renewable generation. We can interpret $\bar{Z}$ as the shared random variable for a specific region. For example, the average solar radiation for a region should be common to every PV output in that region. On the other hand, $ \hat{Z}_i$ can be seen as the individual-level random variable for the particular location of each PV plant $i$, and this random variable can be seen as i.i.d. across different locations.}

	For analytical convenience, we make the following assumptions on $Z_i$:\\
	\noindent \textbf{[A3]} Both $\bar{Z}$ and $\hat{Z}_i$ in \eqref{eq:decomposition} satisfy assumption A1, the $\hat{Z}_i$'s are i.i.d, and are independent of $\bar{Z}$ for all $i$.


	If the correlation is captured as in \eqref{eq:decomposition}, the optimal capacity decision is still symmetric, i.e., $C^*_i = C^*_j, \forall i \neq j$ is a valid solution to \eqref{eq:social_opt}. This is stated in Theorem \ref{th:symmetric_correlation}.
	\begin{theorem}\label{th:symmetric_correlation}
		{If the random variable $Z_i$ is captured as in \eqref{eq:decomposition}} and assumption A3 is satisfied, then the optimal capacity vector that minimizes the planner's social cost is symmetric, i.e., $C_1^* = C_2^* = \cdots = C_N^* = C^*$ when $\gamma_i = \gamma, \forall i$.
	\end{theorem}

	In addition, note that Theorem \ref{th:social_nash} does not require  the i.i.d assumption on $Z_i$. Therefore, we infer that the symmetric solution that minimizes social cost is a Nash equilibrium even when the generation is correlated. In what follows, we further show that correlation does not tamper the efficiency of any Nash equilibria in the capacity game.

	\begin{theorem}
		\label{theorem:correlation}
		Suppose that $({C}^{\diamond}_1, {C}^{\diamond}_2, \ldots, {C}^{\diamond}_N)$ denotes any Nash equilibrium solution in an instance with $N$ producers and $N > \frac{1}{\gamma}$, where $\gamma_i = \gamma, \forall i$. Then, as long as {the random variable $Z_i$, is captured in \eqref{eq:decomposition}}, and assumption A3 is satisfied, we have that:
		$$\sum_{i=1}^N \pi_i({C}^{\diamond}_1, {C}^{\diamond}_2, \ldots, {C}^{\diamond}_N) \leq \gamma \sum_{i=1}^N {C}^{\diamond}_i + \alpha N^{-c},$$

		where $\alpha, c > 0$ are constants that are independent of $N$.
	\end{theorem}


	{Theorem \ref{theorem:correlation} extends the statement in Theorem \ref{maintheorem} from i.i.d. random variables to correlated random variables. This indicates that if the randomness of each producer is captured by an additive model interpreted as the sum of shared randomness and individual-level randomness, then the decentralized market is efficient and that both producers and electricity users benefit from this market.}

	\subsection{Uniqueness of the Symmetric  Equilibrium}\label{sec:uniqueEQ}
	Although our setting could admit many equilibrium solutions, we know that one of these solutions must always be symmetric, i.e., every producer has the same investment level. This solution is of particular interest as it minimizes the social cost.
	We now  show that the symmetric Nash equilibrium $C_1^*, C_2^*, \dots, C_N^*$ is unique in Theorem \ref{theorem:unique}.
	\begin{theorem}\label{theorem:unique}
		Under assumption A1, the symmetric Nash equilibrium in the capacity game \eqref{eq:cap_game} is unique.
	\end{theorem}

	{Theorem \ref{theorem:unique} states that there is only one symmetric Nash equilibrium in the capacity game. This indicates that if the decentralized market is regulated such that each producer behaves similarly in the presence of uncertainty, then it is guaranteed that the competition is both efficient and socially optimal in the investment decision.}

	\subsection{Results on Asymmetric investment price}\label{sec:asym_gamma}

We now extend our results to a more general case when PV providers' investment price is different. This would be the case when installation is dispersed cross areas where unit installation cost is not the same. We consider an arbitrary equilibrium solution of the two-level game with asymmetric investment costs, $(C^{\diamond}_1, C^{\diamond}_2, \ldots, C^{\diamond}_N)$, such that $C^{\diamond}_1 \leq C^{\diamond}_2 \leq \ldots \leq C^{\diamond}_N$. This assumption is clearly without loss of generality. Moreover, we assume that $\gamma_i$ represents the investment cost of the PV whose capacity at equilibrium is given by $C^{\diamond}_i$. Let $\gamma_{min} = \min_{i} \gamma_i$.  {We start with Theorem \ref{theorem_exist}, which guarantees that the game always admits a Nash equilibrium.}

{
	\begin{theorem}\label{theorem_exist}
		Given an instance of the PV game in \eqref{eq:cap_game}, but with asymmetric investment costs $\gamma_1 \geq \gamma_2 \geq \cdots \geq \gamma_N$, there always exists a Nash equilibrium for the capacity game as long as the distributions $(Z_i)_{i=1}^N$ are identical.
	\end{theorem}}

	{Detailed proof is left in Appendix \ref{proof_theorem_asymm_exist}. With this guarantee,} we now introduce the result on the efficiency of the game with asymmetric $\gamma_i$'s.

	\begin{theorem}
		\label{theorem:correlation_asymmetric}
		Suppose that $({C}^{\diamond}_1, {C}^{\diamond}_2, \ldots, {C}^{\diamond}_N)$ denotes any Nash equilibrium solution in an instance with $N$ producers, asymmetric investment costs $(\gamma_1, \ldots, \gamma_N)$ and {$N > \frac{1}{\frac{D_{min}}{D_{max}}\gamma_{\min}}$}. Then, as long as {the random variable $Z_i$, is captured in \eqref{eq:decomposition}}, and assumption A3 is satisfied, we have that:
		$$\sum_{i=1}^N \pi_i({C}^{\diamond}_1, {C}^{\diamond}_2, \ldots, {C}^{\diamond}_N) \leq \sum_{i=1}^N \gamma_{\max} {C}^{\diamond}_i + \alpha N^{-c},$$

		where $\alpha, c > 0$ are constants that are independent of $N$ and $\gamma_{\max} = \max_{i} \gamma_i$.
	\end{theorem}

	Since the actual investment cost incurred by the providers is $\sum_{i=1}^N \gamma_i C^{\diamond}_i$, the above theorem implies that the ratio of the total payment to the investment costs (e.g., the price of anarchy) is at most $\frac{\gamma_{\max}}{\gamma_{\min}}$. This allows us to quantify the efficiency at equilibrium. The efficiency is thus quantified by how much the largest invest price $\gamma_{max}$ deviates from each of the $\gamma_i$'s. When the investment prices are not very different across producers which should be the realistic case, the efficiency is not very far away from the case when $\gamma_i$'s are the same.

\section{Simulation}\label{sec:sim}
In this section, we validate the statements by providing simulation results based on both synthetic data and real PV generation data. We use the symmetric Nash equilibrium as the solution of interest in our simulations.
\subsection{Two-player game}
Let us assume that the generation distribution is uniform, i.e., $Z_i \sim \text{unif}(0, 1)$. Suppose that the investment price is the same for all players, i.e., $\gamma = 0.25$, then following the analysis in Section \ref{sec:mainresults}, we know that the optimal capacity satisfies $C^*_1 = C^*_2$. Assuming that the demand is uniformly distributed between 0.75 and 1.25, we solve the social optimization in \eqref{eq:social_opt} with equal investment price $\gamma$. The optimal solution leads to a total capacity of $C^*_{tot} = C^*_1 + C^*_2 = 1.71$, where $C^*_1 = C^*_2 = 0.86$. The result is shown in Fig. \ref{fig:cost}.

\begin{figure}[!ht]
	\centering
	\includegraphics[width = 0.7\columnwidth]{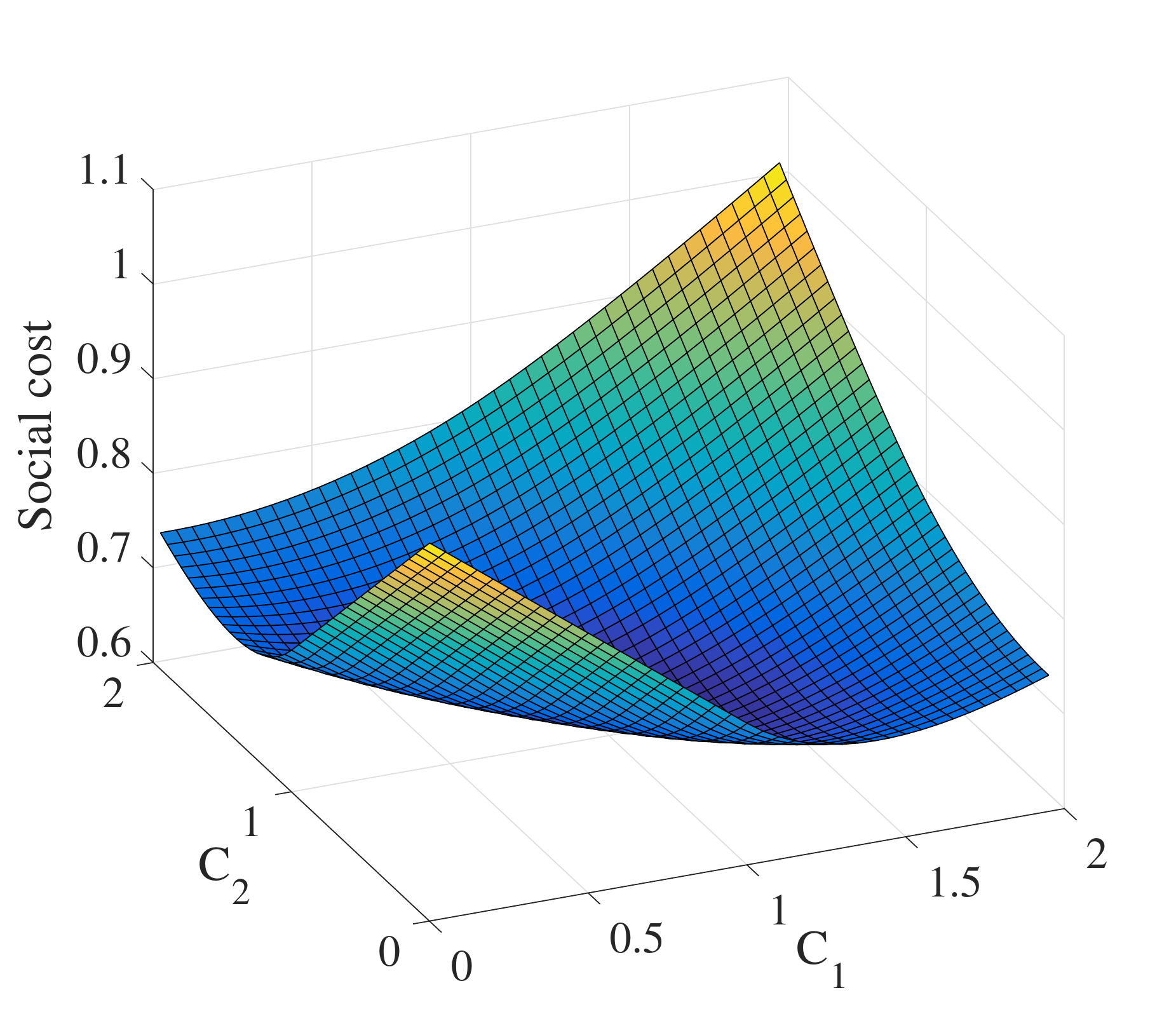}
	\caption{Social cost with respect to total capacity when investment price is the same.}
	\label{fig:cost}
\end{figure}

To verify that $C^*_1 = C^*_2 = 0.86$ is indeed a symmetric Nash equilibrium, we vary the capacity from $C^*_1$ and study how player $1$'s profit changes. The analysis for player 2 proceeds in the same way because of symmetry. We show the result of optimality for player 1 in Fig. \ref{fig:profit} in terms of profit, with a fixed capacity for player 2 where $C_2 = C_2^* = 0.86$.

\begin{figure}[!ht]
	\centering
	\includegraphics[width = 0.6\columnwidth]{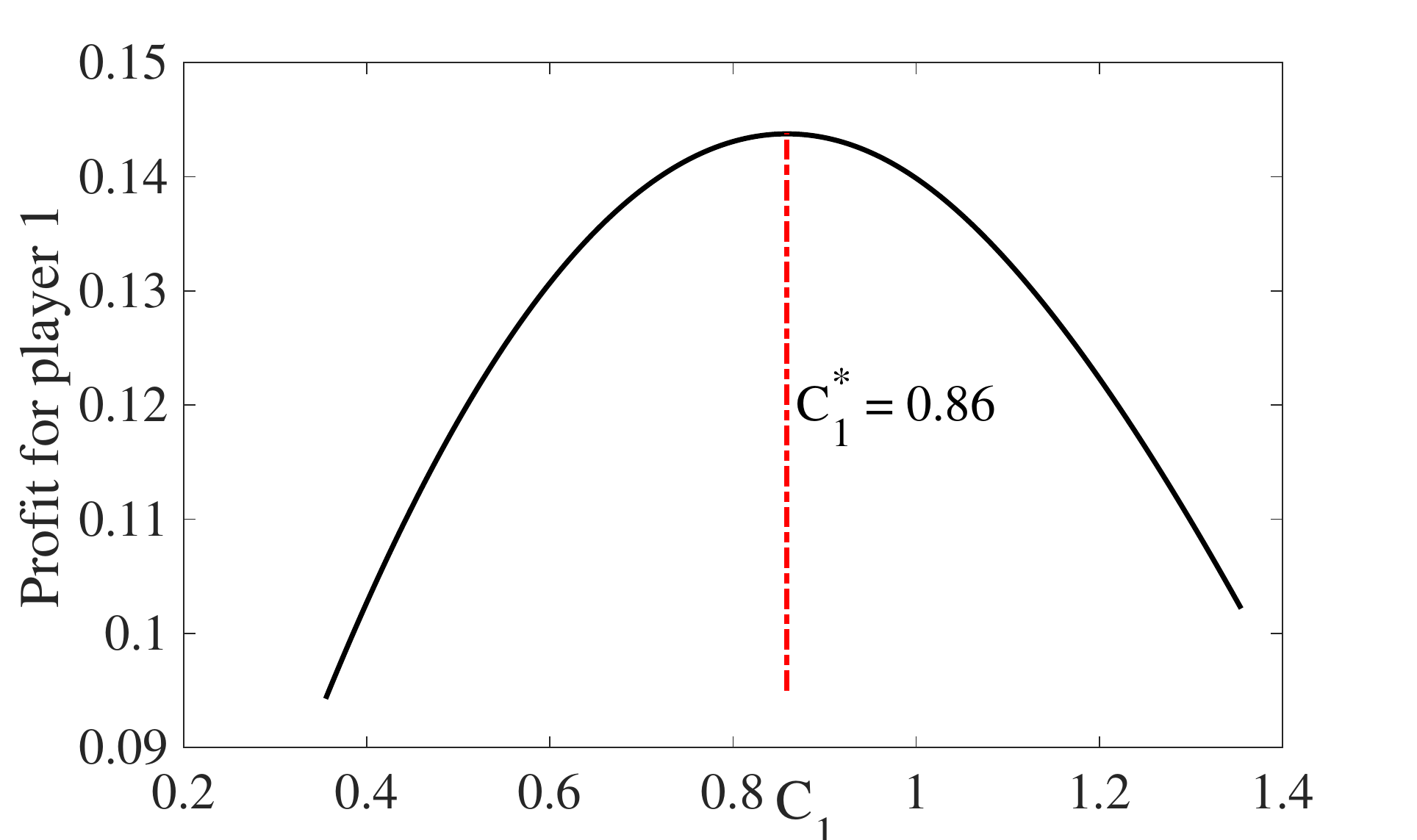}
	\caption{Profit for player 1 when its capacity deviates from $C^*_1$.}
	\label{fig:profit}
\end{figure}

As can be seen from Fig. \ref{fig:profit}, the profit for player 1---when the other player's capacity is fixed at $C^*_2$---peaks at $C_1 = C_1^*$. By symmetry, we can argue that player $2$'s profit is maximized at $C^*_2$ when player $1$'s capacity remains fixed. Therefore, $(C^*_1, C^*_2)$ is indeed a Nash equilibrium as neither player has any incentive to deviate from its investment strategy. In other words, the socially optimal capacity is also a Nash equilibrium for the game shown in \eqref{eq:cap_game}. 

\subsection{N-player game}
To illustrate that the Nash equilibrium is efficient with respect to the metric defined in \eqref{eq:eff}, we need to show that the payment collected from users in the game exactly covers the investment costs of the producers when the number of producers increases. Since the computational complexity grows exponentially with the number of producers (if they are not identical), the simulation is unachievable for large number of asymmetric producers. For illustration purposes we simulate the capacity game with identical players ($\gamma_i = 0.25, \forall i$) with i.i.d. generation (uniform distribution). We then compute the efficiency $\xi$ when there are $10,50,100,150,200,250,300$ players in the game. The results are shown in Fig. \ref{fig:efficiency}.

\begin{figure}[!ht]
	\centering
	\includegraphics[width = 0.6\columnwidth]{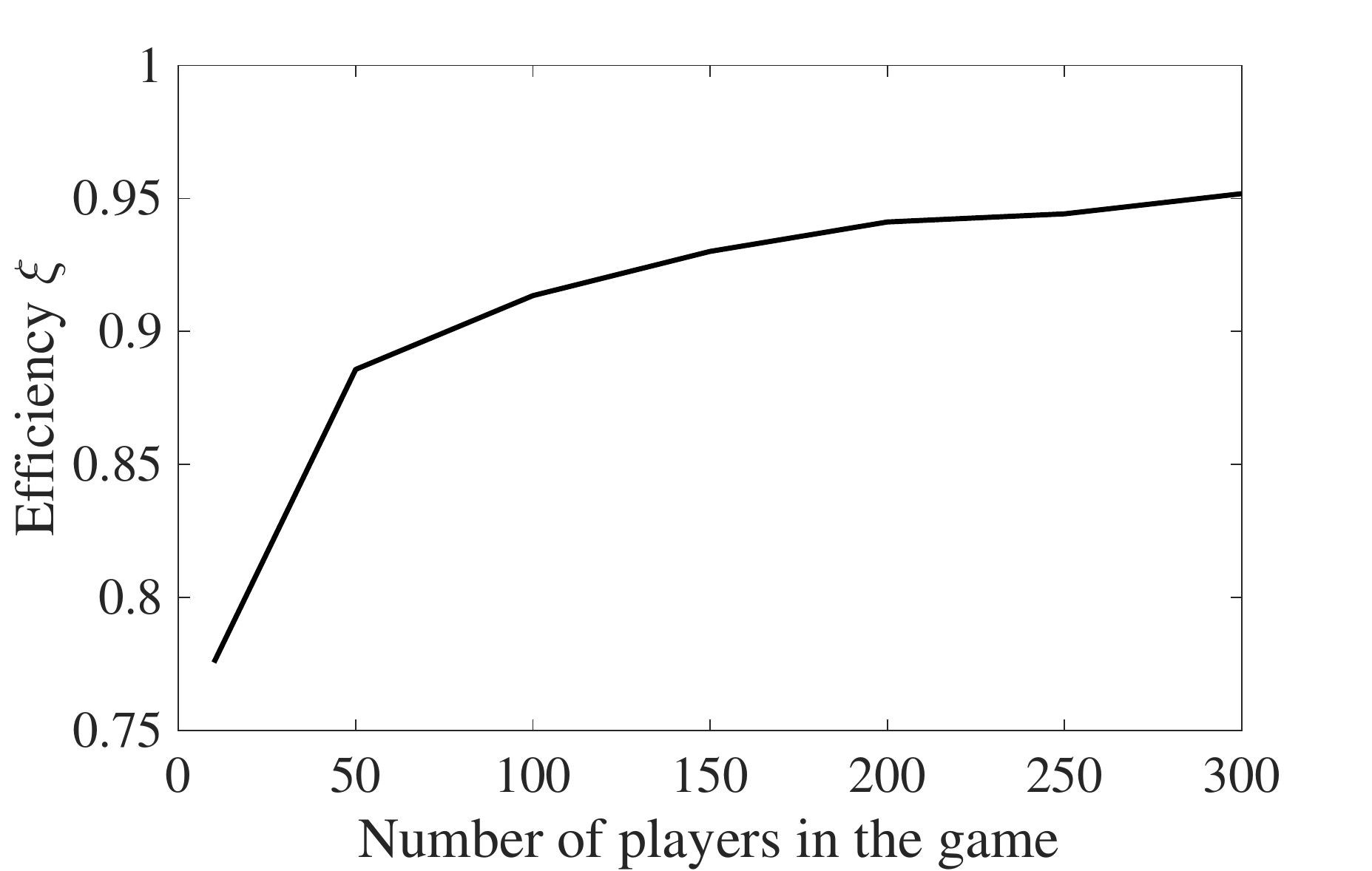}
	\caption{Efficiency of the symmetric Nash equilibrium in the game as a function of number of players.}
	\label{fig:efficiency}
\end{figure}

In Fig. \ref{fig:efficiency}, we see that the efficiency is growing with the number of players in the game. We therefore infer that the competition is healthy as the producers only bid their true costs and do not exploit the consumers of electricity.






\subsection{Case study using real data}
In this section, we simulate the efficiency of the game equilibrium using a real PV generation profile obtained from the National renewable energy laboratory \cite{nrel}. Our data comes from distributed PVs located in California with a 5 minute resolution. Typical PV profiles after normalization are shown in Fig. \ref{fig:PVgen}. {From Fig. \ref{fig:PVgen}, we see that the randomness of PV generation from different locations is strongly correlated. The correlation between those PV profiles is also symmetric across different PV plants, as shown in Fig. \ref{fig:correlation}.
	\begin{figure}[!ht]
		\centering
		\includegraphics[width = 0.9\columnwidth]{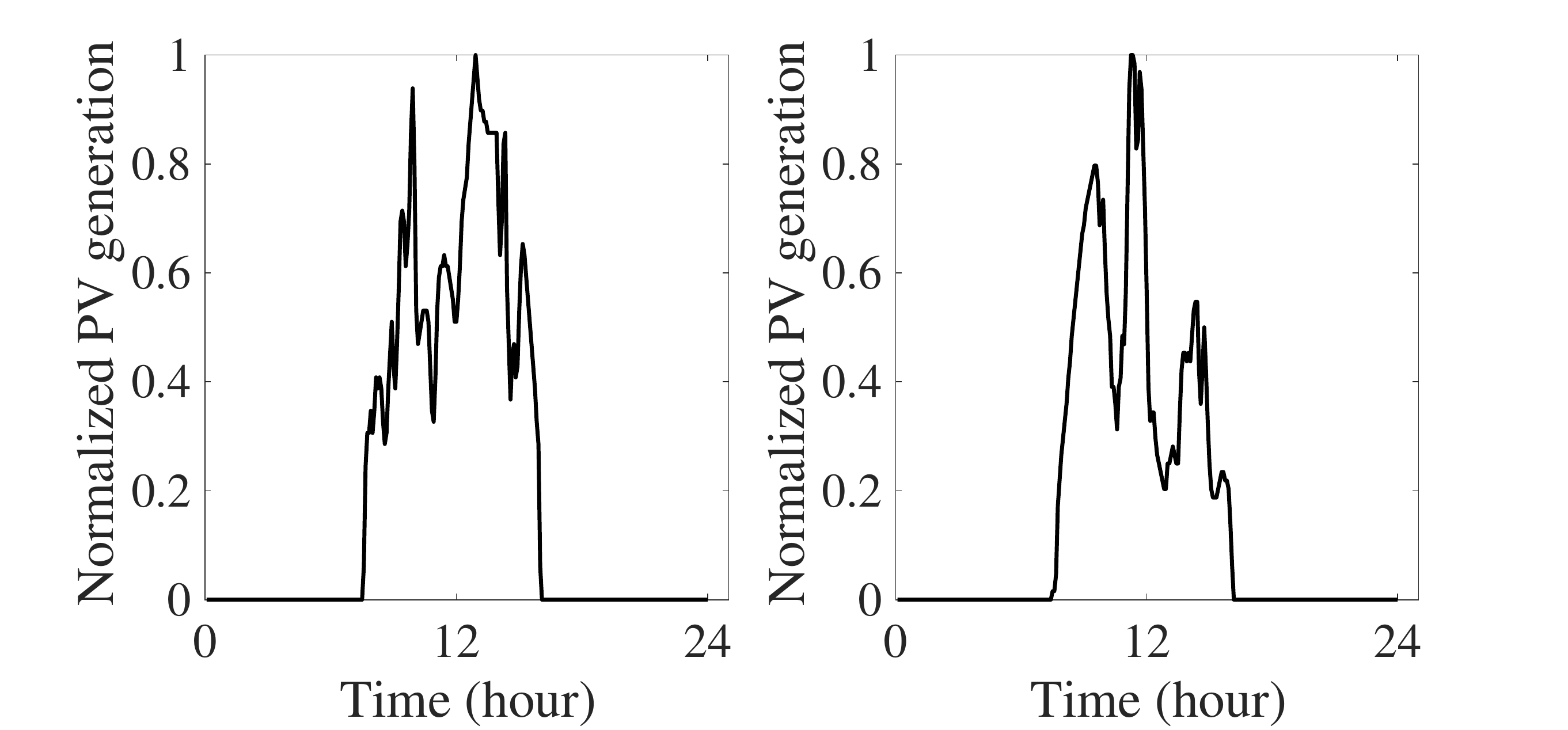}
		\caption{PV generation profile in different locations.}
		\label{fig:PVgen}
	\end{figure}

	\begin{figure}[!ht]
		\centering
		\includegraphics[width = 0.5\columnwidth]{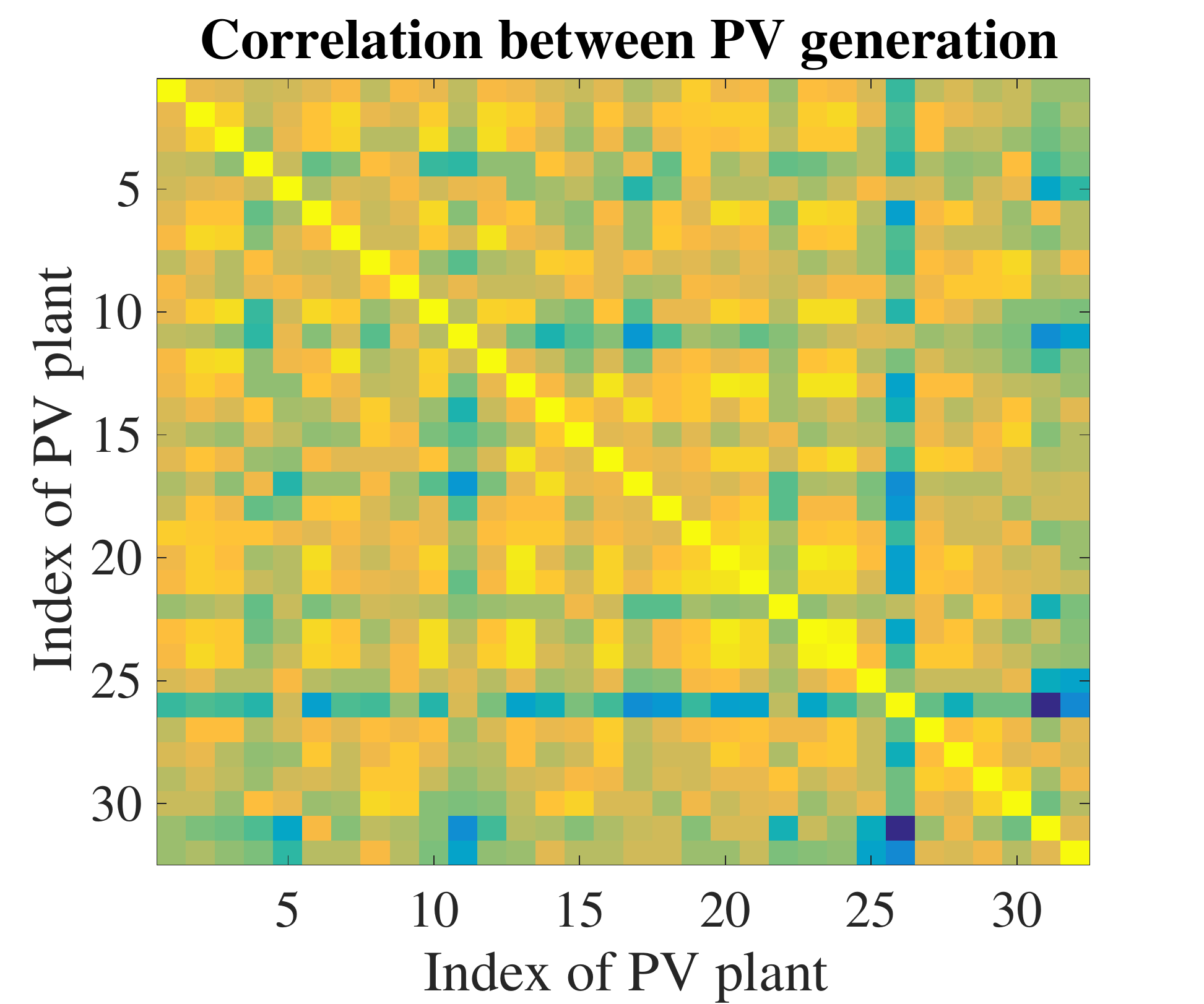}
		\caption{Correlation of PV generation in different locations. A lighter color (yellow) represents stronger correlation and dark colors (blue) represent weak correlation.}
		\label{fig:correlation}
	\end{figure}

	We then use these PV profiles to obtain the game equilibrium as we vary the number of PV participants.  The result is shown in Table \ref{table:eff}, with the assumption that {$\gamma = 0.15$}. As we can see from Table \ref{table:eff}, in the absence of randomness when the producers are assumed to generate energy deterministically, the efficiency is the investment price as described in Equation~\eqref{eq:effdet}. The efficiency of the game with uncertainty improves as the numbers of producers in the market increases.

	\begin{table}[!ht]
		\renewcommand{\arraystretch}{1.3}
		\caption{Game efficiency with different number of producers, when investment price is $0.15$ and demand $D = 5$.}
		\label{table:eff}
		\centering
		\begin{tabular}{|c|c|c|c|}
			\hline
			\bfseries  Number of producers &\bfseries 5 &\bfseries  30 & \bfseries 120 \\
			\hline
			Efficiency of deterministic producers  & 0.15 & 0.15 & 0.15\\
			\hline
			Efficiency of random producers & 0.83 & 0.96 & 0.98\\
			\hline
		\end{tabular}
	\end{table}

	\begin{table}[!ht]
		\renewcommand{\arraystretch}{1.3}
		\caption{The ratio between total capacity and market demand , i.e., ${\sum_i C^*_i}/{D}$, when investment price is $0.15$ and demand $D = 5$.}
		\label{table:total_cap}
		\centering
		\begin{tabular}{|c|c|c|c|}
			\hline
			\bfseries  Number of producers &\bfseries 5 &\bfseries  30 & \bfseries 120 \\
			\hline
			${\sum_i C^*_i}/{D}$ with deterministic producers & 1 & 1 & 1\\
			\hline
			${\sum_i C^*_i}/{D}$ with random producers & 1.26 & 1.32 & 1.30\\
			\hline
		\end{tabular}
	\end{table}

	In addition, in a deterministic game, the total capacity is always the same as the market demand because there is no randomness in generation. In the capacity game with uncertainty, since each producer faces randomness in its own production as well as the random generation from the other producers, the total invested capacity is greater than demand as illustrated in Table. \ref{table:total_cap}. This means that in the capacity game with uncertainty, the total capacity exceeds demand elicits competition among producers. 

\section{Conclusions and Future Work}\label{sec:conclusion}
In this paper, we consider a scenario where many distributed energy resources compete to invest and sell energy in a decentralized electricity market especially when uncertainty is present. Each energy producer optimizes its profit by selling energy. We show that such a competitive game has a Nash equilibrium that coincides with the solution from a social welfare optimization problem. In addition, we show that all Nash equilibria are efficient, in the sense that the collected payment to the energy producers approaches their investment costs. Our statement is validated both by theoretical proofs and simulation studies. Finally, we show that in systems where the investment costs are asymmetric, the inefficiency is at most the ratio of the maximum and minimum investment costs. Of course, in many situations, one can expect different producers have access to similar technologies and therefore, approximately similar investment costs. This can lead to small and bounded levels of inefficiency in practice.

Our work raises a number of intriguing possibilities for future research in this area. First, it is worth noting that the setting studied in this paper admits a multiplicity of equilibrium solutions; we side-step this issue by proving that all of the equilibria are efficient from the perspective of the consumers. In future work, it may be beneficial to study these equilibria in greater detail, particularly the solutions that are more likely to be formed in practice and understand how they compare to the optimal capacity investment strategy. Secondly, this work focuses on the properties of a very specific auction mechanism where each producer sells energy at the exact bid price (i.e., pay-as-bid auction). Understanding how other commonly proposed auction formats---e.g., Cournot mechanism~\cite{ZhangEtAl2015,kian2005bidding}, clock auctions~\cite{maurer2011electricity} ---behave under uncertainty is interesting from a design perspective.

Finally, the central theme in this work---that of leveraging uncertainty to improve outcomes---is applicable to a wide range of systems with some uncertainty in demand or supply. This could include markets in domains such as transportation, communication, cloud computing, etc. Characterizing broad conditions under which uncertainty helps or harms systems with self-interested agents is indeed an important avenue for future research.

\bibliographystyle{IEEEtran}
\bibliography{confbib}

\begin{IEEEbiography}
		[{\includegraphics[width=1in,height=1.25in,clip,keepaspectratio]{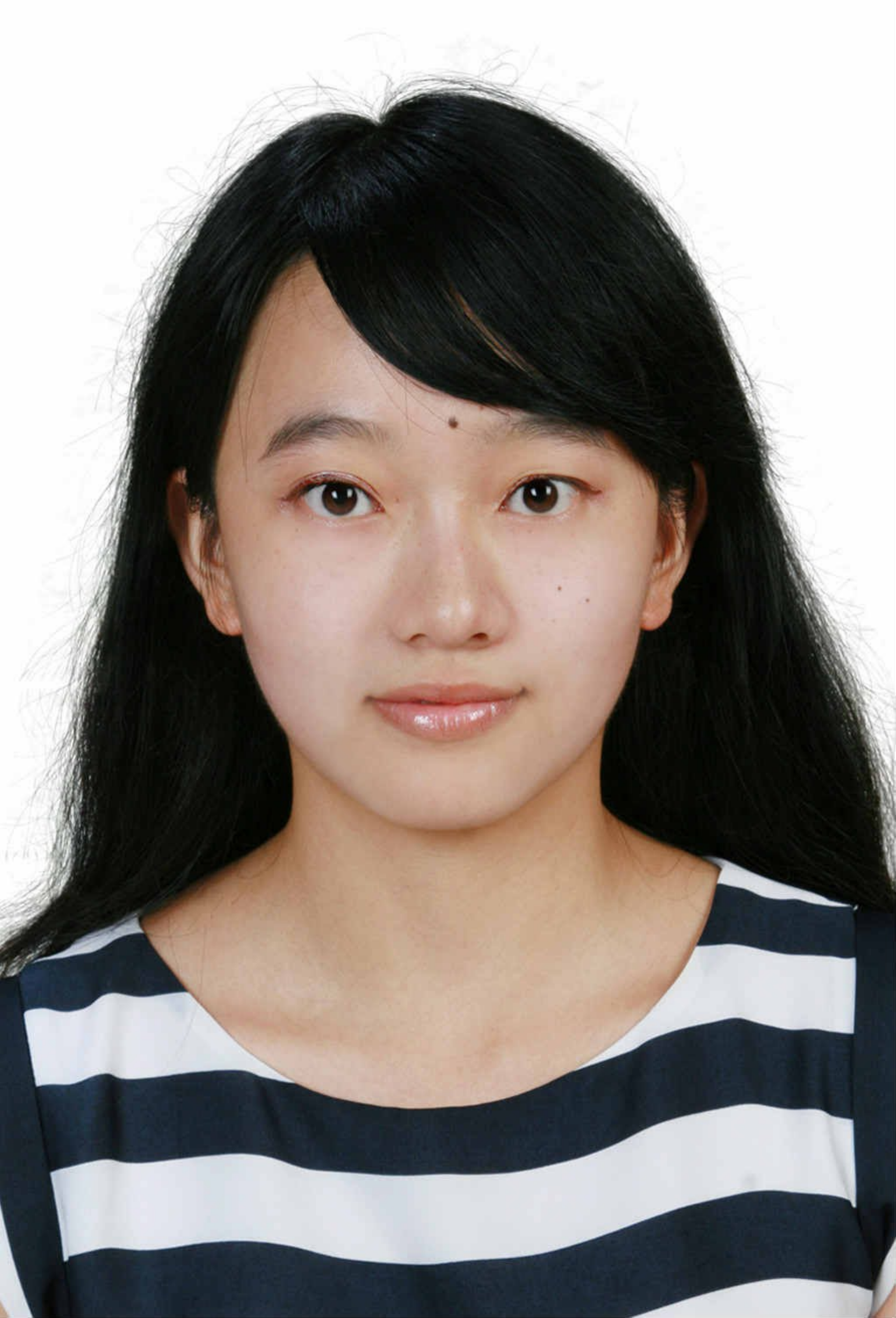}}]{Pan Li} received the B.S. degree in Electrical Engineering, M.S. degree in Systems Engineering
		from Xi'an Jiaotong University, Xi'an, China, in
		2011 and 2014, respectively; and she received her Ph.D. degree in Electrical Engineering from University of Washington, Seattle in 2018. She also holds the engineering diploma from Ecole Centrale de Lille, France. Her research interests included machine learning applications in demand response programs and in distribution networks. She is now research scientist at Facebook, Inc..
	\end{IEEEbiography}
	\begin{IEEEbiography}
	[
	{\includegraphics[width=1in,height=1.25in,clip,keepaspectratio]{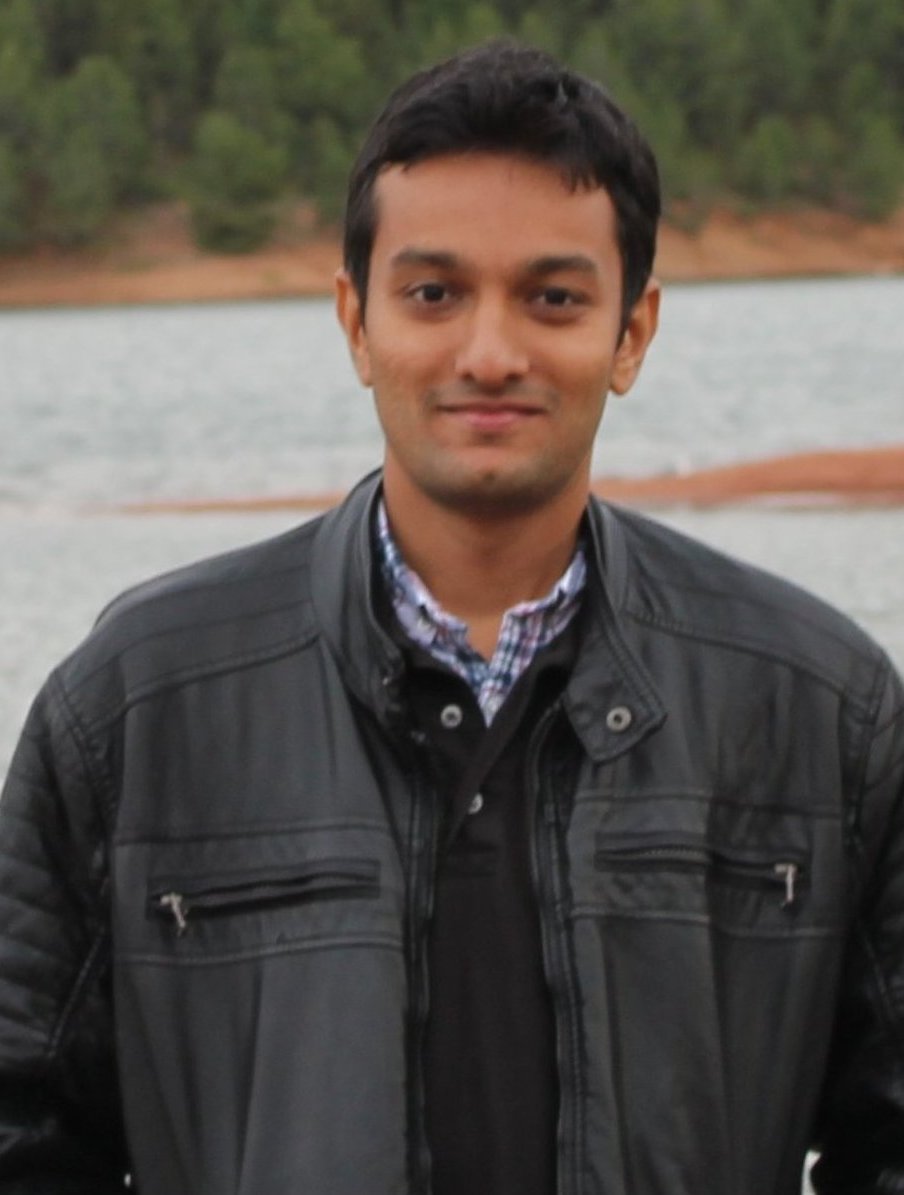}}]{Shreyas Sekar}
	is a Postdoctoral researcher in Electrical Engineering at the University of Washington, Seattle. He received his Ph.D. in Computer Science from Rensselaer Polytechnic Institute in $2017$ and a B.Tech in Electronics and Communications Engineering from the Indian Institute of Technology, Roorkee in $2017$. His research interests are in computational economics and algorithmic game theory. He is a recipient of the $2017$ Robert McNaughton prize for the best graduate student in CS at Rensselaer Polytechnic Institute.
\end{IEEEbiography}
\begin{IEEEbiography}
	[
	{\includegraphics[width=1in,height=1.25in,clip,keepaspectratio]{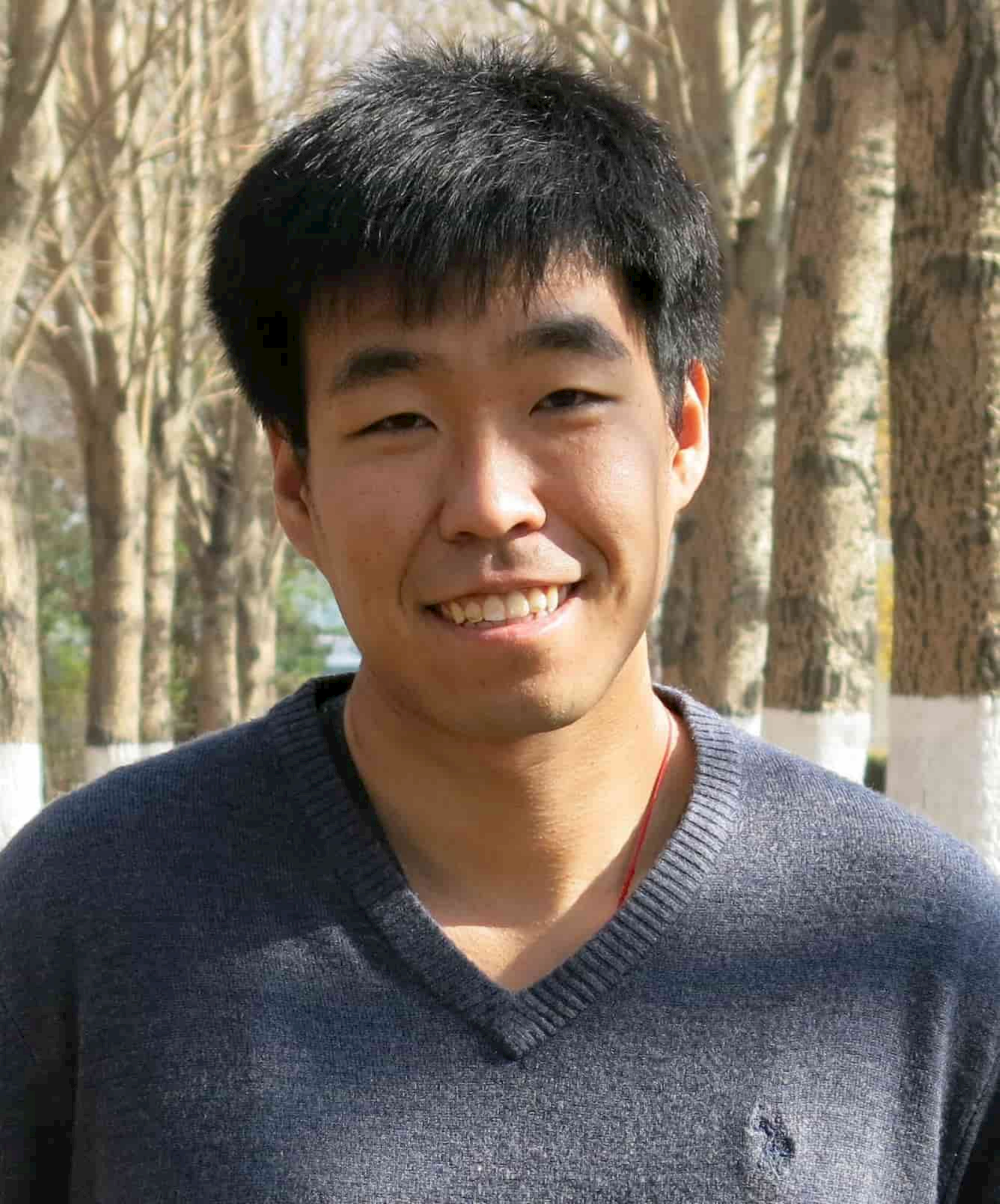}}
	] {Baosen Zhang} received his Bachelor of Applied Science in Engineering Science degree from the University of Toronto in 2008; and his PhD degree in Electrical Engineering and Computer Sciences from University of California, Berkeley in 2013.

	He was a Postdoctoral Scholar at Stanford University, affiliated with the Civil and Environmental Engineering and Management \& Science Engineering. He is currently an Assistant Professor in Electrical and Computer Engineering at the University of Washington, Seattle, WA. His research interests are in power systems and cyberphysical systems.
\end{IEEEbiography}

\clearpage

\appendices
\section{Proof for Proposition \ref{prop:expected_profit}}

We know that the producers adopt a mixed pricing strategy in the equilibrium for the pricing sub-game. Let $l_i, u_i$ denote the lower and upper support of the distribution corresponding to the mixed strategy of producer $i$. From previous results~\cite{AcemogluEtAl09,TaylorEtAl16} and assumption A2, we know that $l_i = l_j$ and $u_i = u_j = 1$ for all $i,j$. {Therefore, let $p$ denote the common lower support (price) for every producer, it is claimed in \cite{TaylorEtAl16} that only one producer can have an atom and it must be at upper support.} Using the basic properties of mixed strategy equilibria (e.g., see~\cite{AcemogluEtAl09}), we can infer that the total payment received by any producer $i$ equals its payment when this producer bids a deterministic price of $p$ and all of the other producers bid according to their mixed strategies in the pricing sub-game equilibrium. Explicitly writing this out, we get
\begin{align}
\pi_i(C_1, C_2, \ldots, C_N) &  = p \E[\min(C_i Z_i, D)]   =  p C_i\E[\min(Z_i, \frac{D}{C_i})] \label{eqn_profitsomeguy} \\
\pi_N(C_1, C_2, \ldots, C_N) & = p \E[\min(C_N Z_N, D)]  =  p C_N \E[\min(Z_i, \frac{D}{C_N})]. \nonumber
\end{align}

{Indeed, observe that when any one player selects a price of $p$, all of the capacity generated by this player must be sold because no other player can bid below this price and the probability that other players bid exactly at this price can be ignored due to the lack of atoms. Using the second equation above, we can explicitly characterize $p$ in terms of the payment received by the producer with the highest capacity investment, i.e.,
$$ p = \frac{\pi_N(C_1, C_2, \ldots, C_N)}{C_N \E[\min(Z_i, \frac{D}{C_N})]} .$$
Substituting the above into Equation~\ref{eqn_profitsomeguy} gives us:
$$ \pi_i(C_1, C_2, \ldots, C_N)  = \pi_N(C_1, C_2, \ldots, C_N) \frac{C_i\E[\min(Z_i, \frac{D}{C_i})]}{C_N \E[\min(Z_N, \frac{D}{C_N})]} .$$}  %

{In order to complete the proof, we need to show that $\pi_N(C_N, C_{-N}) = \E [\min\{(D-\sum_{j \neq N} Z_{j}C_{j})^{+}, Z_NC_N \}]$. The proof utilizes techniques very similar to those adopted in~\cite{AcemogluEtAl09}, so we only sketch the details and highlight the key differences. As argued before, in the equilibrium for the pricing sub-game, each producer $i$ plays a mixed strategy with prices in the range $[p_i,1]$. We first claim that in the equilibrium at most one player has an atom at the upper support price and this can only be producer $N$. The first part of the claim can be proved similar to the proof of Lemma $5$ in~\cite{AcemogluEtAl09}, namely that if multiple producers have an atom at the upper support, then at least one producer can deviate to having an atom at price $1-\epsilon$ instead and strictly improve its profit. For the second part, assume by contradiction that another producer (say $i$) has atom at the upper support. }

{By assumption, only producer $i$ has an atom at the upper support. So, when it bids a deterministic price of $p_i = 1$, the consumers would first consume from the lower priced producers $(j \neq i)$ and only the leftover demand would be satisfied by $i$. Therefore, we can infer that the payment received by $i$ equals: $\pi_i(\bm{C}) = \E [\min\{(D-\sum_{j \neq i} Z_{j}C_{j})^{+}, Z_iC_i \}.]$ Then, as per~\eqref{eqn_paymentasymmetric}, we have that $\pi_N(\bm{C}) = \frac{\pi_i(\bm{C}) C_N G_N}{C_iG_i},$ where $G_j = E[\min(Z_j, \frac{D}{C_j})]$ for all $j$. We have that:
\begin{align}
\pi_N(\bm{C}) & =  \E [\min\{(D-\sum_{j \neq i} Z_{j}C_{j})^{+}, Z_iC_i \}] \frac{C_N G_N}{C_iG_i} \notag\\
& <  \E [\min\{(D-\sum_{j \neq N} Z_{j}C_{j})^{+}, Z_NC_N \}] \frac{G_N}{G_i} \notag\\
& \leq \E [\min\{(D-\sum_{j \neq N} Z_{j}C_{j})^{+}, Z_NC_N \}]. \label{eqn_profitN_prop1}
\end{align}
The strict inequality comes from the fact that $C_i < C_N$ and that $Z_i, Z_N$ are identically distributed (conditional on $Z_{-i,N})$. The final inequality comes from the fact that $G_N \leq G_i$ since $C_i < C_N$. However,~\eqref{eqn_profitN_prop1} is a lower bound on the producer $N$'s payment if it deviates to bidding a deterministic price of $p_N = 1$, and so, we have a contradiction since producers cannot strictly improve their payments by deviating at equilibrium. In summary, we have that only producer $N$ can have an atom at the upper support. As we argued before, since no other producer as an atom at the upper support, when producer $N$ bids $p_N = 1$, it only gets the leftover demand and therefore, its payment is given by $\pi_N(\bm{C}) = \E [\min\{(D-\sum_{j \neq N} Z_{j}C_{j})^{+}, Z_NC_N \}.] \hfill \qed$}

\section{Proof for Theorem \ref{th_symm2}}
Suppose that at the optimum solution that minimizes Equation~\eqref{eq:social_opt}, the aggregate capacity investment by the producers is $C^*_{tot}$, and let $C^*_1 = C^*_2 = \ldots = C^*_N = \frac{C^*_{tot}}{N} \geq 0$. Then, in order to prove Theorem~\eqref{th_symm2}, it is sufficient to show that for any capacity $C_1, C_2, \ldots, C_N \geq 0$ with $\sum_{i=1}^N C_i = C^*_{tot}$, the following equation is satisfied:
$$ \E[(D-\sum_{i=1}^N Z_i C^*_i)^+] \leq \E[(D-\sum_{i=1}^N Z_i C_i)^+].$$

In fact, using the transformation that for any capacity vector $(C'_1, \ldots, C'_N)$, $\E[(D-\sum_{i=1}^N Z_i C'_i)^+] = \E D - \E[\min(D,\sum_{i=1}^N Z_i C'_i)]$, the above equation can be rewritten as:
\begin{equation}
\label{eqn_objective_rewrite}
\E[\min(D,\sum_{i=1}^N Z_i C^*_i)] \geq  \E[\min(D,\sum_{i=1}^N Z_i C_i)].
\end{equation}

So, to prove Theorem~\ref{th_symm2}, it is indeed sufficient to prove Equation~\eqref{eqn_objective_rewrite}. To prove this, we introduce Proposition \ref{prop:inequality} and Proposition \ref{prop:equality}.

\begin{proposition}\label{prop:inequality}
	Let us consider the following definitions:
	\begin{itemize}
		\item $X_1 \coloneqq Z_1 \frac{C_1}{N} + Z_2 \frac{C_2}{N} + \ldots + Z_N \frac{C_N}{N}$,
		\item $X_2 \coloneqq Z_1 \frac{C_2}{N} + Z_2 \frac{C_3}{N} + \ldots + Z_N \frac{C_1}{N}$
		\item $\ldots$
		\item $X_N \coloneqq Z_1 \frac{C_N}{N} + Z_2 \frac{C_1}{N} + \ldots + Z_N \frac{C_{N-1}}{N}$
	\end{itemize}
	That is:
	\begin{equation}\label{def:X}
	X_i = \sum_{j=1}^N Z_j \frac{C_{i+j-1}}{N},
	\end{equation}
	where $i+j-1$ is computed modulo $N$. 

	Then:
	\begin{equation}
	\label{eqn_opt_realcase}
	\E[\min(D,\sum_{i=1}^N X_i)] \geq \sum_{i=1}^N \E[\min(\frac{D}{N}, X_i)].
	\end{equation}
\end{proposition}

\begin{proof}
	We prove the inequality by proving that inequality holds for each realization of $X_1 = x_1, X_2 = x_2, \dots, X_N = x_N$ and each realization of $D = d$. Denote the whole index set by $\mathcal{N}$. In each realization, there are the following four scenarios:
	\begin{itemize}
		\item Each of $x_i$ is smaller than $\frac{d}{N}$. In this case, $\min(d,\sum_{i=1}^N x_i) = \sum_{i=1}^N x_i$, and $\min(\frac{d}{N}, x_i) = x_i$. So equality holds.
		\item Each of $x_i$ is bigger than $\frac{d}{N}$. In this case, $\min(d,\sum_{i=1}^N x_i) = d$, and $\min(\frac{d}{N}, x_i) = \frac{d}{N}$. So equality again holds.
		\item $x_j, j \in \mathcal{J} \subseteq \mathcal{N}$ is bigger than $\frac{d}{N}$, the rest are smaller than $\frac{d}{N}$, but $\sum_{i = 1}^{N} x_i \leq d$. In this case, $\min(d,\sum_{i=1}^N x_i) = \sum_{i=1}^N x_i$. The RHS of \eqref{eqn_opt_realcase} reduces to $\frac{d}{N}|\mathcal{J}| + \sum_{i \in \mathcal{N}\setminus \mathcal{J}}x_i \leq \sum_{i=1}^N x_i$. Therefore, the inequality holds.
		\item $x_j, j \in \mathcal{J}\subseteq \mathcal{N}$ is smaller than $\frac{d}{N}$, the rest are bigger than $\frac{d}{N}$, but $\sum_{i = 1}^{N} x_i \geq d$. The RHS of \eqref{eqn_opt_realcase} reduces to $\frac{d}{N}(N-|\mathcal{J}|) + \sum_{j \in \mathcal{J}}x_j \leq d$. Therefore, the inequality holds.
	\end{itemize}

	In all cases, we have that:
	\begin{equation}
	\E[\min(D,\sum_{i=1}^N X_i)] \geq \sum_{i=1}^N \E[\min(\frac{D}{N}, X_i)].
	\end{equation}
\end{proof}

\begin{proposition}\label{prop:equality}
	With the assumptions in Proposition \ref{prop:inequality}, we have:
	\begin{equation}
	\E[\min(\frac{D}{N}, X_i)] =  \E[\min(\frac{D}{N}, X_1)], \forall i \in \mathcal{N}.
	\end{equation}
\end{proposition}

To prove Proposition \ref{prop:equality}, we introduce Fact \ref{fact:exchangeable}. It is based on exchangeability of independent random variables, and that $Z_i$'s' are i.i.d. copies. We refer to \cite{wiki_exch} for interested users and omit the proof here.
\begin{fact}\label{fact:exchangeable}
	If $Z_i $'s are i.i.d. random variables,  then:
	\begin{equation}
	f(z_1, z_2, \dots, z_N) = f(z_{S_1}, z_{S_2}, \dots, z_{S_N}),
	\end{equation}
	where $S_1, S_2, \dots, S_N$ is a permutation of $1, 2, \dots, S_N$, and $f(\cdot)$ is the density function.
\end{fact}
Now we proceed to prove Proposition \ref{prop:equality}.
\begin{proof}[Proof of Proposition \ref{prop:equality}]
	Using Fact \ref{fact:exchangeable}, we show that
	$\E[\min(\frac{D}{N}, X_i)]$ is the same for all $i$:
	\begin{equation}
	\begin{aligned}
	\E[\min(D/N, X_i)]  & = \E[\min(D/N,\sum_{j=1}^N Z_j \frac{C_{i+j-1}}{N})] \\
	& \stackrel{(a)}{=} \E[\min(D/N,\sum_{j=1}^N Z_{i+j-1} \frac{C_{i+j-1}}{N})] \\
	& \stackrel{(b)}{=} \E[\min(D/N,\sum_{k = 1}^N Z_{k} \frac{C_{k}}{N})] \\
	& =  \E[\min(D/N, X_1)],
	\end{aligned}
	\end{equation}
	where $(a)$ is based on the observation that $[i, i+1, \dots, i+N-1], \mod{N}, \forall i$ is a permutation of $[1, 2, \dots, N]$, and $(b)$ is the result of rearranging $C_{i+j-1}$.
\end{proof}



Now, we are ready to prove ~\eqref{eqn_objective_rewrite}. Let $(X_i)^N_{i=1}$ be as defined in the statement of Proposition~\ref{prop:inequality}. The LHS of \eqref{eqn_objective_rewrite} can be written as:
{
\begin{equation}\label{eq:final_proof}
\begin{aligned}
\E[\min(D,\sum_{i=1}^N Z_i C^*_i)]
& =  \E[\min(D,\sum_iZ_i C^*_{tot}/N)]  \\
& \stackrel{(a)}{=}  \E[\min(D,\sum_iZ_i \frac{\sum_j C^*_j}{N})]  \\
& \stackrel{(b)}{=}  \E[\min(D,\sum_iZ_i X_i)]  \\
& \stackrel{(c)}{\geq}  \sum_{i=1}^N \E[\min(\frac{D}{N}, X_i)]\\
& \stackrel{(d)}{=}  \sum_{i=1}^N \E[\min(\frac{D}{N}, X_1)] \\
& = N \E[\min(\frac{D}{N}, X_1)] \\
& = \E[\min(D, \sum_{i=1}^N Z_i C_i)],
\end{aligned}
\end{equation}}
{where we decompose $C^*_{tot}$ into individual $C^*_j$'s in $(a)$ and use definition of $X_i$ in Proposition \ref{prop:inequality}. Inequality $(c)$ is again based on Proposition \ref{prop:inequality}  and $(d)$ is based on Proposition \ref{prop:equality}}. This concludes the proof. \hfill \qed

\section{Proof for Theorem \ref{th:social_nash}}
Suppose that there are $N$ producers in the market,
and suppose that the optimal capacity from solving \eqref{eq:social_opt} is denoted by $C^*_1, C^*_2,..., C^*_N$, we argue that $C^*_1, C^*_2,..., C^*_N$ is a Nash equilibrium for the capacity game in \eqref{eq:cap_game}.

We prove the equilibrium for player 1, and the same argument holds for any of the rest players. To show this, we rewrite $C^*_1$ as the following:
{
\begin{equation}
\begin{aligned}
C^*_1 & = arg\min_{C_1}\ \gamma C_1 + \gamma \sum_{i = 2}^{N} C^*_i + \E \{(D-\sum_{i = 2}^{N} Z_iC_i - Z_1C_1)^+ \}\\
& = arg\min_{C_1}\ \gamma C_1 + D - \E \min \{D, \sum_{i = 2}^{N} Z_iC_i  + Z_1C_1\} \\
& = arg\min_{C_1}\ \gamma C_1 -\E \min \{D - \sum_{i = 2}^{N} Z_iC_i, Z_1C_1\} \\
& = arg\min_{C_1} \ \gamma C_1 - \E \min \{(D - \sum_{i = 2}^{N} Z_iC_i)^+ , Z_1C_1\} \\ & \qquad \qquad \qquad -  \min \E \{(D - \sum_{i = 2}^{N} Z_iC_i)^- , Z_1C_1\} \\
& = arg\min_{C_1} \ \gamma C_1 - \E \min \{(D - \sum_{i = 2}^{N} Z_iC_i)^+ , Z_1C_1\} \\
 & - \E (D - \sum_{i = 2}^{N} Z_iC_i)^-  \\
& = arg\max_{C_1} \  \E \min \{(D - \sum_{i = 2}^{N} Z_iC_i)^+ , Z_1C_1\}  - \gamma C_1 \\
& = C^{\diamond}_1,
\end{aligned}
\end{equation}
}
which characterizes the optimal solution to the game depicted in \eqref{eq:cap_game}. \hfill \qed

\section{Proof for Theorem \ref{th:symmetric_correlation}}

Similar to the proof for Theorem \ref{th_symm2}, we need to show that \eqref{eqn_objective_rewrite} is true when $Z_i = \bar{Z} + \hat{Z}_i$ as given in \eqref{eq:decomposition}. The proof boils down to show that Proposition \ref{prop:inequality} and Proposition \ref{prop:equality} are true under such asssumption on correlation. Note that the proof for Proposition \ref{prop:inequality} does not require that $Z_i$'s to be i.i.d., therefore naturally carries over.
To show Proposition \ref{prop:equality}, we need Lemma \ref{lemma:exchangeable_2}. Then these two propositions validate \eqref{eqn_objective_rewrite}, which concludes the proof.
\begin{lemma}\label{lemma:exchangeable_2}
	If $Z_i = \bar{Z} + \hat{Z}_i$ as \eqref{eq:decomposition}, then:
	$$f(z_1, z_2, \dots, z_N) = f(z_{S_1}, z_{S_2}, \dots, z_{S_N}),$$ where $S_1, S_2, \dots, S_N$ is a permutation of $1, 2, \dots, S_N$, and $f(\cdot)$ is the density function.
\end{lemma}

\begin{proof}
	\begin{equation*}
	\begin{aligned}
	& f(z_1, z_2, \dots, z_N)\\
	 & = f(\bar{z}+\hat{z}_1, \bar{z}+\hat{z}_2, \dots, \bar{z}+\hat{z}_N)  \\
	& = \int_{\bar{z}} f( \{\bar{z}+\hat{z}_1, \bar{z}+\hat{z}_2, \dots, \bar{z}+\hat{z}_N\} | \bar{Z} = \bar{z}) f(\bar{Z} = \bar{z}) d{\bar{z}} \\
	& \stackrel{(a)}{=}  \int_{\bar{z}} f( \bar{z}+\hat{z}_1, \bar{z}+\hat{z}_2, \dots, \bar{z}+\hat{z}_N) f(\bar{Z} = \bar{z}) d{\bar{z}} \\
	& \stackrel{(b)}{=}  \int_{\bar{z}} f( \bar{z}+\hat{z}_{S_1}, \bar{z}+\hat{z}_{S_2}, \dots, \bar{z}+\hat{z}_{{S_N}}) f(\bar{Z} = \bar{z}) d{\bar{z}} \\
	& =  f(z_{S_1}, z_{S_2}, \dots, z_{S_N}),
	\end{aligned}
	\end{equation*}
	where $(a)$ is based on the assumption that $\bar{Z}$ is independent of $\hat{Z}_i$ (assumption A3), and $(b)$ is based on Fact \ref{fact:exchangeable}.
\end{proof}

\section{Proof for Theorems \ref{maintheorem} and Theorem \ref{theorem:correlation}}

\subsection{Berry-Esseen Theorem}
The following Lemma is useful to facilitate the proofs of Theorem \ref{maintheorem} and Theorem \ref{theorem:correlation}. It relates the behavior of the mean of independent random variables to a standard Gaussian distribution in terms of CDF.
\begin{lemma}[Berry-Esseen Theorem]\label{lemma:berry}
	There exists a positive constant $\alpha$, such that if $X_1$, $X_2$, $\dots$, $X_N$, are independent random variables with $\E(X_i) = 0$, $\E(X_i^2) = \sigma^2 > 0$, and $\E(|X_i|^3) = \rho_i < \infty$, and if we define $S_{N}=\frac{\sum_i X_i}{\sqrt{\sum_i \sigma_i^2}}$, then $F_N$, the cumulative distribution function of
	$S_{N}$
	is close to $\Phi$, the CDF of the standard Gaussian distribution. This is mathematically interpreted as:
	\begin{equation}
	|F_N(x) - \Phi(x)| \leq \alpha \psi,
	\end{equation}
	where $\psi = (\sum_i \sigma_i^2)^{-\frac{1}{2}} \max_{1 \leq i \leq N} \frac{\rho_i}{\sigma_i^2}$.
\end{lemma}

\subsection{Some useful lemmas}
Before the detailed proof, let us visit some useful propositions and lemmas that assist the proofs for Theorem \ref{maintheorem} and Theorem \ref{theorem:correlation}. In what follows, we assume without loss of generality that given any solution $(C_1, C_2, \ldots, C_N)$, it must be the case that $C_1 \leq C_2 \leq \ldots \leq C_N$.

\begin{proposition}\label{lem_derivation}
	The partial derivative of \\$\pi_N(\bm{C}_{-N}, C_N) = \E (\min(D - \sum_{j = 1}^{N-1}C_iZ_i)^+, C_N)$ is:
	\begin{equation}
	\frac{\partial \pi_N(\bm{C}_{-N}, C_N)}{\partial C_N} = \E \left[\mathbbm{1}\{(D - \sum_{j = 1}^{N-1}C_jZ_j)^+ \geq C_N Z_N\right]Z_N \},
	\end{equation}
	and for all $i \neq N$.
	\begin{equation}
	\frac{\partial \pi_N(\bm{C}_{-N}, C_N)}{\partial C_i} = \E \left[\mathbbm{1}\{0 \leq (D - \sum_{j = 1}^{N-1}C_jZ_j) \leq C_N Z_N\}(-Z_i)\right].
	\end{equation}
\end{proposition}

Based on Proposition \ref{lem_derivation}, we have the following lemma on the optimality of the symmetric Nash equilibrium of the game.
\begin{lemma}\label{lem_symmetric_property}
	If the invested capacity {at Nash equilibrium} is symmetric, i.e., $C_1 = C_2 = \dots = C_N = C$, then:
	\begin{itemize}
		\item $\gamma = \E \left[\mathbbm{1}\left\{(D-C\sum_{j \neq i}^NZ_j)^+\geq C Z_i\right\} Z_i \right]$, where $\mathbbm{1}\{{\cdot}\}$ is the indicator function and takes value 1 if the argument is true, otherwise takes value 0.
		\item $NC \leq \E D/\gamma$.
	\end{itemize}
\end{lemma}
\begin{proof}

	We begin by proving the first part. Suppose that capacities are the same, i.e., $C_1 = C_2 = \dots = C_N = C$, {the profit for each producer in the capacity game is}:
	{
	\begin{equation}
	 \E \left[\min \left( (D - C \sum_{j\neq i}^{N}Z_j)^+, CZ_i  \right) \right] - \gamma C.
	\end{equation}}

	The optimality of a player $i$ in the game is captured as the following:
	\begin{equation}
	\frac{\partial}{\partial C_i} \E\left[\min\left((D-C\sum_{j \neq i}^NZ_j)^+,C_iZ_i\right)\right]- \gamma =0.
	\end{equation}
	Differentiating in the expectation with respect to each individual $C_i$ and based on Proposition \ref{lem_derivation}, we have:
	\begin{equation}\label{eqn:gamma}
	\gamma=\E \left[\mathbbm{1}\left\{(D-C\sum_{j \neq i}^NZ_j)^+\geq C Z_i\right\} Z_i \right],
	\end{equation}
	where $\mathbbm{1}\{{\cdot}\}$ is the indicator function and takes value 1 if the argument is true, otherwise takes value 0.

	Then we proceed to prove the second part, i.e., $NC \leq \E D/\gamma$.

	When there is a social planner making centralized decision as described in \eqref{eq:social_opt}, the total payment from the electricity consumers in the system is
	\begin{equation}\label{eqn:socialmin}
	\min_{C_i, \forall i= 1, 2, \dots, N} \E\left[ (D-  \sum_{i=1}^N C_i Z_i)^+\right] + \gamma \sum_{i=1}^NC_i,
	\end{equation}
	where each site should have the same optimal invested capacity $C_1 = C_2 = \dots = C_N = C$ as discussed in Section \ref{sec:mainresults}. This also coincides with the symmetric Nash equilibrium in the capacity game.

	To show that $NC$ is a bounded by a constant, assuming differentiability and based on \eqref{eqn:socialmin}, we know
	\begin{align}
	\gamma N&=\E\left[1(D \geq C \sum_{i =1}^{N} Z_i) \sum_{i =1}^{N} Z_i\right] \\
	&=\E\left[1(\sum_{i =1}^{N} Z_i \leq D/C) \sum_{i =1}^{N} Z_i \right] \\
	& \leq \E\left[1(\sum_{i =1}^{N} Z_i \leq D/C) D/C \right] \\
	& \leq \E D/C \label{eqn_symmetricupper},
	\end{align}
	rearranging, we get $NC \leq \E D/\gamma$.
\end{proof}

Based on Lemma \ref{lem_symmetric_property}, we now present two lemmas on arbitrary Nash equilibria of the game in Lemma \ref{lem_capupperbound} and Lemma \ref{lem_expectationgames}.
\begin{lemma}
	\label{lem_capupperbound}
	Given any equilibrium solution of the two-level game $(C^{\diamond}_1, C^{\diamond}_2, \ldots, C^{\diamond}_N)$, it must be the case that $C^{\diamond}_1 \leq \frac{\E D}{\gamma N}$ and $C^{\diamond}_N \leq \frac{\E D}{\gamma}$.
\end{lemma}
\begin{proof}
	Assume by contradiction that the first part of the above statement is not true and there exists an equilibrium solution $(C^{\diamond}_1, C^{\diamond}_2, \ldots, C^{\diamond}_N)$ such that $\frac{\E D}{\gamma N} < C^{\diamond}_1 \leq \ldots \leq C^{\diamond}_N$. Recall the formula for the payment received by producer $N$, i.e., \\$\pi_N(C^{\diamond}_1, C^{\diamond}_2, \ldots, C^{\diamond}_N) = E[(D-\sum_{i\neq N} C^{\diamond}_i Z_i)^+, C^{\diamond}_NZ_N)]$. Since this is an equilibrium solution, the derivative of this payment must equal the investment cost $\gamma$. More specifically, using the expression for the derivative that was previously derived in Equation~\eqref{eqn:gamma}, we have that

	\begin{align*}
	\left(\frac{d}{dC}\pi_N(C^{\diamond}_1, C^{\diamond}_2, \ldots, C)\right)_{C = C^{\diamond}_N} & = \gamma \\
	\implies \E \left[\mathbbm{1}\left\{(D-\sum_{i \neq N}C^{\diamond}_iZ_i)^+\geq C^{\diamond}_N Z_N\right\} Z_N \right] & = \gamma.
	\end{align*}

	Now, since the symmetric equilibrium solution $(C^*, \ldots, C^*)$ also satisfies this condition, we have that:
	$$\E \left[\mathbbm{1}\left\{(D-\sum_{i \neq N}C^*Z_i)^+\geq C^* Z_N\right\} Z_N \right]  = \gamma.$$

	Further, recall that in the symmetric equilibrium $C^* \leq \frac{\E D}{\gamma N}$ as derived in Equation~\eqref{eqn_symmetricupper}. Since $C^{\diamond}_1 > \frac{\E D}{\gamma N}$, this implies that $C^{\diamond}_1 > C^*$. Finally, let $\mathcal{E}$ denote the set of events\footnote{For our purposes, an event is a tuple of instantiations of the i.i.d random variables $(Z_1, Z_2, \ldots, Z_N)$ and $D$ satisfying the required condition} $\mathbbm{1}\left\{(D-\sum_{i \neq N}C^{\diamond}_iZ_i)^+\geq C^{\diamond}_N Z_N\right\}$ and let $\mathcal{E}^*$ denote the events satisfying\\$\mathbbm{1}\left\{(D-\sum_{i \neq N}C^*Z_i)^+\geq C^* Z_N\right\}$. Since $C^* < C^{\diamond}_1 \leq C^{\diamond}_2 \leq \ldots \leq C^{\diamond}_N$, it is not hard to deduce that $\mathcal{E} \subset \mathcal{E}^*$. {Indeed for  any (non-zero) instantiation $(Z_1, \ldots, Z_N)$ when $(D-\sum_{i \neq N}C^*Z_i)^+ > 0$,  we have that $(D-\sum_{i \neq N}C^{\diamond}_iZ_i)^+ < (D-\sum_{i \neq N}C^*Z_i)^+$ and $C^{\diamond}_NZ_N > C^*Z_N$.} Therefore, we get that:

	\begin{align*}
	&\E \left[\mathbbm{1}\left\{(D-\sum_{i \neq N}C^{\diamond}_iZ_i)^+\geq C^{\diamond}_N Z_N\right\} Z_N \right]  < \\ & \quad \E \left[\mathbbm{1}\left\{(D-\sum_{i \neq N}C^*Z_i)^+\geq C^* Z_N\right\} Z_N \right] = \gamma ,
	\end{align*}

	which is a contradiction.

	Next, we prove that at equilibrium $C^{\diamond}_N \leq \frac{\E D}{\gamma}$. The proof is somewhat similar and once again, proceeds by contradiction. Suppose that $C^{\diamond}_N > \frac{\E D}{\gamma}$. {Now, we have:
	\begin{align*}
	\gamma & = \E \left[\mathbbm{1}\left\{(D-\sum_{i \neq N}C^{\diamond}_iZ_i)^+\geq C^{\diamond}_N Z_N\right\} Z_N \right] \\
	& \leq \E \left[\mathbbm{1}\left\{D \geq C^{\diamond}_N Z_N\right\} Z_N \right] \\
	& < \E \left[\mathbbm{1}\left\{D \geq \frac{\E D}{\gamma} Z_N\right\} Z_N \right]  \\
	& = \E \left[\mathbbm{1}\left\{\gamma \geq \frac{\E D}{D}Z_N \right\} Z_N \right]  \\
	& \leq  \E \left[\mathbbm{1}\left\{\gamma \geq \frac{\E D}{D}Z_N \right\} \frac{ D}{\E D}\gamma \right]  \\
	& = \frac{1}{\E D} \gamma \E \{ \mathbbm{1}\left\{ \gamma \geq \frac{\E D}{D}Z_N\right\} D \}\\
	&  \leq \gamma,
	\end{align*}
}
	which is an obvious contradiction. \end{proof}

\begin{lemma}
	\label{lem_expectationgames}
	Suppose that $ C_i \leq C_j \leq \frac{\E D}{k}$ for some $k \leq 1$. Then, we have that:

	$$\E[\min(Z_j, \frac{D}{C_j})] \leq \E[\min(Z_i, \frac{D}{C_i})].$$
	$$ \E[\min(Z_j, \frac{D}{C_j})] \geq q k \E[Z_i].$$
\end{lemma}
\begin{proof}
	The first part is easy to see. For any instantiation of $Z_j$ and $D$, we have that $\min(Z_j, D/C_j) \leq \min(Z_j, D/C_i)$ since $C_i \leq C_j$. Taking the expectation, and changing the variable from $Z_j$ to $Z_i$ (these variables have the same marginal distribution due to either i.i.d assumption, or follows \eqref{eq:decomposition}), we get that:

	$$ \E[\min(Z_j, \frac{D}{C_j})] \leq \E[\min(Z_j, \frac{D}{C_i})] = \E[\min(Z_i, \frac{D}{C_i})].$$
	The second part of the lemma can be proved as follows: once again fix any instantiation of $Z_j$, we have that:
	{
	\begin{equation}
	\begin{aligned}
	\E_{D|Z_j}\min(Z_j, \frac{D}{C_j}) & \geq E_{DD|Z_j} \mathbbm{1}\{D \geq \E D\} \min(Z_j, \frac{D}{C_j})\\
	& \geq q \min(Z_j, k) \\
	& \geq q k\min(Z_j,1) \\
	& = q kZ_j
	\end{aligned}
	\end{equation}
}

	Taking the expectation, we get the required result. \end{proof}

Last, we present a lemma on the bound for integrating on a standard Gaussian distribution.
\begin{lemma}\label{lemma_gaussian_diff}
	Let $\Phi(\cdot)$ denote the CDF for standard Gaussian distribution, i.e., zero mean and unit variance. Then:
	\begin{align}
	\Phi(x) - \Phi(y) \leq \frac{1}{\sqrt{2\pi}}(x - y), x>y.
	\end{align}
\end{lemma}

Lemma \ref{lemma_gaussian_diff} is a direct observation based on the density function $f(x)$ of standard Gaussian random variable, i.e., $f(x) = \frac{1}{\sqrt{2\pi}} e^{-x^2}$ which has a maximum value of $\frac{1}{\sqrt{2\pi}}$.

\subsection{Proof for Theorem \ref{theorem:correlation} and  Theorem \ref{maintheorem}}\label{subsec:proof_maintheorem}

Now we proceed to prove Theorem \ref{theorem:correlation} and Theorem \ref{maintheorem}. Note that Theorem \ref{maintheorem} is a special case of \ref{theorem:correlation} when $\bar{Z} = 0$, in the following proof, we assume that $Z_i = \bar{Z} + \hat{Z}_i$ as in \eqref{eq:decomposition}.

To avoid lengthy notation, let us define $G_i = \E[\min(Z_i, D/{C}^{\diamond}_i)]$ for all $1 \leq i \leq N$. Consider the payment received by the producer with the smallest investment, which happens to be ${C}^{\diamond}_1$. As per Equation~\eqref{eqn_paymentasymmetric}, this equals:
\begin{equation}
	\begin{aligned}
	\pi_1({C}^{\diamond}_1, \ldots, {C}^{\diamond}_N) & = \pi_N({C}^{\diamond}_1,  \ldots, {C}^{\diamond}_N) \frac{{C}^{\diamond}_1 G_1 }{{C}^{\diamond}_N G_N }\\
	& = \pi_N({C}^{\diamond}_1,  \ldots, {C}^{\diamond}_N) \frac{{C}^{\diamond}_1 E[Z_1] }{{C}^{\diamond}_N G_N }.
	\end{aligned}
\end{equation}

{The second equation comes from our assumption that $N > \frac{1}{\frac{D_{min}}{D_{max}}\gamma}$}. Therefore by Lemma~\ref{lem_capupperbound}, we have that $C_1 \leq D$ and $G_1 = \E[\min(Z_1, 1)] = \E[Z_1]$. In what follows, we will continue to use $G_1 = \E[Z_1]$ for consistency but remark that $G_1$ is a constant that is independent of ${C}^{\diamond}_1$. The total profit made by this producer is $\pi_1({C}^{\diamond}_1, \ldots, {C}^{\diamond}_N) - \gamma {C}^{\diamond}_1$. Since this is an equilibrium solution, we have that $\left(\frac{d}{dC}\pi_1(C, {C}^{\diamond}_2, \ldots, {C}^{\diamond}_N)\right)_{C={C}^{\diamond}_1} = \gamma$. Expanding the differentiation term, we get that:
\begin{align}
& \frac{G_1}{{C}^{\diamond}_NG_N} \Bigg(\pi_N({C}^{\diamond}_1,  \ldots, {C}^{\diamond}_N) \nonumber\\
& \quad  - {C}^{\diamond}_1 \E[  \mathbbm{1} \left\{ 0 \leq D-\sum_{i \neq N}{C}^{\diamond}_iZ_i \leq {C}^{\diamond}_NZ_N  \right\}  Z_1\Bigg) = \gamma \label{eqn_derivativeprofit1}.
\end{align}

In the above equation, we used the fact that: $\frac{d}{dC} \pi_N(C,  \ldots, {C}^{\diamond}_N) = \E[ \mathbbm{1} \left\{ 0 \leq D-CZ_1 +\sum_{i=2}^{N-1}{C}^{\diamond}_iZ_i \leq {C}^{\diamond}_NZ_N  \right\} - Z_1 ]$. Rearranging Equation~\eqref{eqn_derivativeprofit1}, we get an upper bound for the payment made to producer $N$, namely
\begin{equation}
	\begin{aligned}
	\pi_N({C}^{\diamond}_1,  \ldots, {C}^{\diamond}_N) & = \gamma \frac{{C}^{\diamond}_N G_N}{G_1}  \\
	& + {C}^{\diamond}_1 \E[ \mathbbm{1} \{ 0 \leq D-\sum_{i \neq N}{C}^{\diamond}_iZ_i \leq {C}^{\diamond}_NZ_N  \}  Z_1]  \label{eqn_derivativeprofit2}.
	\end{aligned}
\end{equation}

Fix some constant $\kappa$. The rest of the proof proceeds in two cases:

(Case I: $\sum_{i=1}^{N-1} \frac{{C}^{\diamond}_i}{{C}^{\diamond}_N} \leq \kappa N^{\frac{3}{4}}$)

Intuitively, this refers to the case where the investments are rather asymmetric---i.e., the investment by the `larger producers' is significantly bigger than that by the `smaller producers'. Note that in Equation~\eqref{eqn_derivativeprofit2}, $Z_1 \leq 1$. Therefore, we get:
\begin{align*}
\pi_N({C}^{\diamond}_1,  \ldots, {C}^{\diamond}_N) & \leq \gamma \frac{{C}^{\diamond}_N G_N}{G_1} + {C}^{\diamond}_1 \E[ \mathbbm{1} \{ 0 \leq D-\sum_{i \neq N}{C}^{\diamond}_iZ_i \leq {C}^{\diamond}_NZ_N  \} ] \\
& = \gamma \frac{{C}^{\diamond}_N G_N}{G_1} + {C}^{\diamond}_1 Pr\left(  0 \leq D-\sum_{i \neq N}{C}^{\diamond}_iZ_i \leq {C}^{\diamond}_NZ_N  \right) \\
& \leq \gamma \frac{{C}^{\diamond}_N G_N}{G_1} + {C}^{\diamond}_1 .
\end{align*}

Recall from Lemma~\ref{lem_capupperbound} that in any equilibrium solution, we must have that ${C}^{\diamond}_1 \leq \frac{\E D}{\gamma N}$. Substituting this above, we get that
$$\pi_N({C}^{\diamond}_1,  \ldots, {C}^{\diamond}_N)  \leq \gamma \frac{{C}^{\diamond}_N G_N}{G_1}+ \frac{\E D}{\gamma N}.$$

Next, observe that for any $i \neq N$, we can apply 
Proposition \ref{prop:expected_profit} to obtain an upper bound on its profit, namely that:
\begin{align}
\pi_i({C}^{\diamond}_1,  \ldots, {C}^{\diamond}_N)  & \leq \pi_N({C}^{\diamond}_1,  \ldots, {C}^{\diamond}_N) \frac{{C}^{\diamond}_i G_i}{{C}^{\diamond}_N G_N} \nonumber \\
& \leq \left(\gamma \frac{{C}^{\diamond}_N G_N}{G_1} + \frac{\E D}{\gamma N} \right) \frac{{C}^{\diamond}_i G_i}{{C}^{\diamond}_N G_N}  \nonumber\\
& = \gamma \frac{{C}^{\diamond}_i G_i}{G_1}  + \frac{\E D {C}^{\diamond}_i G_i  }{ \gamma N{C}^{\diamond}_N G_N} \nonumber\\
& \leq \gamma {C}^{\diamond}_i  + \frac{\E D {C}^{\diamond}_i \E[Z_i]  }{ q \gamma N {C}^{\diamond}_N \gamma \E[Z_N]} \label{eqn_profitupperboundall2}\\
& = \gamma {C}^{\diamond}_i  + \frac{\E D {C}^{\diamond}_i }{ q \gamma^2 N{C}^{\diamond}_N }  \label{eqn_profitupperboundall}.
\end{align}

Equations~\eqref{eqn_profitupperboundall2} and~\eqref{eqn_profitupperboundall} were derived using Lemma~\ref{lem_expectationgames}, namely we used the simple properties that $(i)$ $G_i \leq G_1$, $(ii)$ $G_i \leq \E[Z_i]$, and $(iii)$ $G_N \geq q \gamma \E[Z_N]$, and finally the fact that $\E[Z_i] = \E[Z_N]$.

Summing up \eqref{eqn_profitupperboundall} over all $i$ including $i=N$, we get that
$$\sum_{i=1}^N \pi_i({C}^{\diamond}_1,  \ldots, {C}^{\diamond}_N) \leq \gamma \sum_{i=1}^N {C}^{\diamond}_i + \frac{\E D}{N q \gamma^2} \left(\sum_{i=1}^{N}\frac{{C}^{\diamond}_i}{{C}^{\diamond}_N}\right). $$

Of course, as per our assumption, we have that $\sum_{i=1}^{N-1}\frac{{C}^{\diamond}_i}{{C}^{\diamond}_N} \leq \kappa N^{\frac{3}{4}}.$ Substituting this above, we get that
\begin{align*}
\sum_{i=1}^N \pi_i({C}^{\diamond}_1,  \ldots, {C}^{\diamond}_N) & \leq \gamma \sum_{i=1}^N {C}^{\diamond}_i + \frac{\E D}{Nq\gamma^2} \left(\kappa N^{\frac{3}{4}} + 1)\right).\\
& \leq \gamma \sum_{i=1}^N {C}^{\diamond}_i + \frac{\E D}{q\gamma^2} \left(\kappa N^{-\frac{1}{4}} + \frac{1}{N}\right).
\end{align*}
This proves the theorem statement for the case where $\sum_{i=1}^{N-1} \frac{{C}^{\diamond}_i}{{C}^{\diamond}_N} \leq \kappa N^{\frac{3}{4}}$. Note that the $\frac{1}{N}$ term can be incorporated into the constant $\alpha$, without affecting any of the asymptotic bounds.

(Case II: $\sum_{i=1}^{N-1} \frac{{C}^{\diamond}_i}{{C}^{\diamond}_N} > \kappa N^{\frac{3}{4}}$)

Let us go back to Equation~\eqref{eqn_derivativeprofit1} and consider the term \\$B = \E[ \mathbbm{1} \left\{ 0 \leq D-\sum_{i \neq N}{C}^{\diamond}_iZ_i \leq {C}^{\diamond}_NZ_N  \right\}  Z_1]$. We will now obtain a tighter upper bound on this quantity conditional upon $\sum_{i=1}^{N-1} \frac{{C}^{\diamond}_i}{{C}^{\diamond}_N} > \kappa N^{\frac{3}{4}}$. First note that applying the Cauchy-Schwarz inequality, we can get a lower bound on the sum-of-squares, i.e.,
\begin{equation}
\label{eqn_cslower}
\sum_{i=1}^{N-1} (\frac{{C}^{\diamond}_i}{{C}^{\diamond}_N})^2 \geq \frac{1}{N-1}(\sum_{i=1}^{N-1} {C}^{\diamond}_i/{C}^{\diamond}_N)^2 \geq \frac{\kappa^2 N^{3/2}}{N-1} \geq 2\kappa^2 N^{\frac{1}{2}}.
\end{equation}
The final simplification comes from the fact that $N \geq 2$ (since we assumed $N > \frac{1}{\gamma} \geq 1$). Next, we have that:
{
\begin{align*}
B & \leq \E[ \mathbbm{1} \left\{ 0 \leq D-\sum_{i \neq N}{C}^{\diamond}_iZ_i \leq {C}^{\diamond}_NZ_N  \right\}] \\
& = Pr_{\cdot|D}\left( D - {C}^{\diamond}_N Z_N \leq \sum_{i \neq N}{C}^{\diamond}_iZ_i \leq D  \right) \\
& = \E_D\Pr_{\cdot|D}\left( \frac{D}{{C}^{\diamond}_N} - Z_N \leq \sum_{i \neq N}\frac{{C}^{\diamond}_i}{{C}^{\diamond}_N}Z_i \leq \frac{D}{{C}^{\diamond}_N}  \right)\\
& \leq \E_D \Pr_{\cdot|D}\left( \frac{D}{{C}^{\diamond}_N} - 1 \leq \sum_{i \neq N}\frac{{C}^{\diamond}_i}{{C}^{\diamond}_N}Z_i \leq \frac{D}{{C}^{\diamond}_N}  \right).
\end{align*}
}



Using the fact that $Z_i = \hat{Z}_i + \bar{Z}$, we can write out the probability more explicitly:
{
\begin{equation}\label{eq:simplified}
\begin{aligned}
&  \E_D\Pr_{\cdot|D}\left( \frac{D}{{C}^{\diamond}_N} - 1 \leq \sum_{i \neq N}\frac{{C}^{\diamond}_i}{{C}^{\diamond}_N}Z_i \leq \frac{D}{{C}^{\diamond}_N}  \right)  \\
& =  \E_D\Pr_{\cdot|D}\left( \frac{D}{{C}^{\diamond}_N} - 1 \leq \sum_{i \neq N}\frac{{C}^{\diamond}_i}{{C}^{\diamond}_N}(\hat{Z}_i + \bar{Z}) \leq \frac{D}{{C}^{\diamond}_N}  \right).  \\
\end{aligned}
\end{equation}
}

Let us denote $Z_0 =\sum_{i \neq N}\frac{{C}^{\diamond}_i}{{C}^{\diamond}_N}  \bar{Z}$ and $Z' = \sum_{i \neq N}\frac{{C}^{\diamond}_i}{{C}^{\diamond}_N}\hat{Z}_i $. From the assumption A3 along with \eqref{eq:decomposition} we know that $\bar{Z}$ is independent of $\hat{Z}_i$, therefore $Z_0$ is independent of $Z'$.  
The probability in \eqref{eq:simplified} can be written as an integral over the possible values for $Z_0$ with $f_{Z_0}(\cdot)$ being the probability density function for the random variable $Z_0$. So, we have:
{
\begin{equation}\label{eq:correlated}
\begin{aligned}
& \E_D\Pr_{\cdot|D}\left( \frac{D}{{C}^{\diamond}_N} - 1 \leq \sum_{i \neq N}\frac{{C}^{\diamond}_i}{{C}^{\diamond}_N}Z_i \leq \frac{D}{{C}^{\diamond}_N}  \right)  \\
& = \E_D\Pr_{\cdot|D}\left( \frac{D}{{C}^{\diamond}_N} - 1 \leq Z_0 + Z' \leq  \frac{D}{{C}^{\diamond}_N}\right) \\
& \stackrel{(a)}{=} \E_D\int_{0}^{\frac{D}{{C}^{\diamond}_N}} f_{Z_0}(z)Pr\left(\frac{D}{{C}^{\diamond}_N} - 1 \leq Z_0 + Z' \leq \frac{D}{{C}^{\diamond}_N}| Z_0 = z \right) dz \\
& \stackrel{(b)}{=} \E_D\int_{0}^{\frac{D}{{C}^{\diamond}_N}} f_{Z_0}(z)Pr\{ \frac{D}{{C}^{\diamond}_N} - 1 - z  \leq Z' \leq \frac{D}{{C}^{\diamond}_N}-z \}dz, \\
\end{aligned}
\end{equation}
}
where $(a)$ uses the property of conditional probability, and $(b)$ is based on the fact that $Z_0$ and $Z'$ are independent.



Since $Z' = \sum_{i \neq N}\frac{{C}^{\diamond}_i}{{C}^{\diamond}_N}\hat{Z}_i$, where $\hat{Z}_i$'s are i.i.d. random variables and $\E \hat{Z}_i = \hat{\mu}$, $\E (\hat{Z}_i - \E\hat{Z}_i)^2 = \hat{\sigma}^2 > 0$, $\E |\hat{Z}_i - \E \hat{Z}_i|^3 = \hat{\rho} < \infty$, this variable has a mean $\mu' = \hat{\mu} \sum_{i=1}^{N-1}{C}^{\diamond}_i/{C}^{\diamond}_N \geq \hat{\mu} \kappa N^{\frac{3}{4}}$. Similarly, the variance of $Z'$ can be written as $(\sigma')^2 = \hat{\sigma}^2 \sum_{i=1}^{N-1}(\frac{{C}^{\diamond}_i}{{C}^{\diamond}_N})^2$. Now, applying Equation~\eqref{eqn_cslower}, we get the following lower bound for variance:
\begin{equation}\label{eq_sigma_prime}
(\sigma')^2 \geq 2 \hat{\sigma}^2 \kappa^2 N^{\frac{1}{2}}.
\end{equation}

Note that $\frac{{C}^{\diamond}_i}{{C}^{\diamond}_N}\hat{Z}_i$ are independent random variables because $\hat{Z}_i$'s are i.i.d.. What is more, we know that $\frac{{C}^{\diamond}_i}{{C}^{\diamond}_N}\hat{Z}_i$ has mean $\bar{\mu}_i = \frac{{C}^{\diamond}_i}{{C}^{\diamond}_N} \hat{\mu}$, non negative variance $\bar{\sigma}_i^2 = (\frac{{C}^{\diamond}_i}{{C}^{\diamond}_N})^2 \hat{\sigma}^2$ and finite centered third moment $\bar{\rho}_i = (\frac{{C}^{\diamond}_i}{{C}^{\diamond}_N})^3 \hat{\rho}$. Denote $S_N = \frac{Z' - \sum_i \bar{\mu}_i}{\sqrt{\sum_i \bar{\sigma}_i^2}}$ and let $F_N$ denote the CDF of $S_N$. We rewrite the probability of interest as the following:
\begin{equation}\label{eq:proofberry}
\begin{aligned}
&  \Pr_{\cdot|D}\left( \frac{D}{{C}^{\diamond}_N} - 1 - z\leq \sum_{i \neq N}\frac{{C}^{\diamond}_i}{{C}^{\diamond}_N}\hat{Z}_i \leq \frac{D}{{C}^{\diamond}_N} -z \right) \\
= &  \Pr_{\cdot|D}\left( \frac{D}{{C}^{\diamond}_N} - 1 - z\leq Z' \leq \frac{D}{{C}^{\diamond}_N}  - z\right) \\
= &  \Pr_{\cdot|D}\left( \frac{\frac{D}{{C}^{\diamond}_N} - 1 - z - \sum_i \bar{\mu}_i}{\sqrt{\sum_i \bar{\sigma}_i^2}} \leq \frac{Z' - \sum_i \bar{\mu}_i}{\sqrt{\sum_i \bar{\sigma}_i^2}} \leq  \frac{\frac{D}{{C}^{\diamond}_N} - z \sum_i \bar{\mu}_i}{\sqrt{\sum_i \bar{\sigma}_i^2}} \right) \\
= &  \left\{F_N\left(\frac{\frac{D}{{C}^{\diamond}_N} - z - \sum_i \bar{\mu}_i}{\sqrt{\sum_i \bar{\sigma}_i^2}}\right) - F_N\left(\frac{\frac{D}{{C}^{\diamond}_N} - 1 - z - \sum_i \bar{\mu}_i}{\sqrt{\sum_i \bar{\sigma}_i^2}}\right) \right\}. \\
\end{aligned}
\end{equation}

We can now apply the Berry-Esseen Theorem from Lemma~\ref{lemma:berry}	 to get an upper bound:
\begin{equation}\label{eq:upperbound}
\begin{aligned}
& Pr_{\cdot|D}\left( \frac{D}{{C}^{\diamond}_N} - 1 - z\leq \sum_{i \neq N}\frac{{C}^{\diamond}_i}{{C}^{\diamond}_N}\hat{Z}_i \leq \frac{D}{{C}^{\diamond}_N} -z \right) \\
= & F_N\left(\frac{\frac{D}{{C}^{\diamond}_N} - z - \sum_i \bar{\mu}_i}{\sqrt{\sum_i \bar{\sigma}_i^2}}\right) - F_N\left(\frac{\frac{D}{{C}^{\diamond}_N} - 1 - z - \sum_i \bar{\mu}_i}{\sqrt{\sum_i \bar{\sigma}_i^2}}\right) \\
\stackrel{(a)}{\leq} & \Phi \left(\frac{\frac{D}{{C}^{\diamond}_N} - z - \sum_i \bar{\mu}_i}{\sqrt{\sum_i \bar{\sigma}_i^2}}\right) - \Phi \left(\frac{\frac{D}{{C}^{\diamond}_N} - 1- z - \sum_i \bar{\mu}_i}{\sqrt{\sum_i \bar{\sigma}_i^2}}\right) \\
+ & 2\alpha (\sum_i \bar{\sigma}_i^2)^{-\frac{1}{2}}\max_i \frac{\bar{\rho}_i}{\bar{\sigma}_i^2}\\
\stackrel{(b)}{\leq} & \Phi \left(\frac{\frac{D}{{C}^{\diamond}_N} - z - \sum_i \bar{\mu}_i}{\sqrt{\sum_i \bar{\sigma}_i^2}}\right) - \Phi \left(\frac{\frac{D}{{C}^{\diamond}_N} - 1- z - \sum_i \bar{\mu}_i}{\sqrt{\sum_i \bar{\sigma}_i^2}}\right) \\
+ & 2\alpha (2\bar{\sigma}^2\kappa^2N^{-\frac{1}{4}})\max_i(\frac{{C}^{\diamond}_i}{{C}^{\diamond}_N}  \frac{\bar{\rho}}{\bar{\sigma}^2}) \\
\stackrel{(c)}{\leq} & \frac{1}{\sqrt{2\pi}}\frac{1}{\sqrt{\sum_i \bar{\sigma}_i^2}} + 2\alpha (2\bar{\sigma}^2\kappa^2N^{-\frac{1}{4}})(\frac{\bar{\rho}}{\bar{\sigma}^2}) \\
\stackrel{(d)}{\leq} & \kappa' N^{-\frac{1}{4}},
\end{aligned}
\end{equation}
where $\Phi(\cdot)$ is the CDF for standard Gaussian distribution. Inequality $(a)$ applies Berry Esseen Theorem to $F_N(x)$ at $x = \frac{\frac{D}{{C}^{\diamond}_N} - z - \sum_i \bar{\mu}_i}{\sqrt{\sum_i \bar{\sigma}_i^2}}$ and $x = \frac{\frac{D}{{C}^{\diamond}_N} - 1- z -\sum_i \bar{\mu}_i}{\sqrt{\sum_i \bar{\sigma}_i^2}}$. Inequality $(b)$ is based on \eqref{eq_sigma_prime}. Inequality $(c)$ is based on Lemma \ref{lemma_gaussian_diff}. Inequality $(c)$ also depends on the fact that $\frac{{C}^{\diamond}_i}{{C}^{\diamond}_N}<1, \forall i$. Lastly, $(d)$ uses the fact that $\frac{1}{\sum_i \bar{\sigma}_i^2} \leq \frac{1}{\sqrt{2\sigma^2\kappa^2N^{\frac{1}{2}}}}$ from \eqref{eq_sigma_prime}, and rewrites the constants into $\kappa'$ for brevity.

Plugging the above upper bound back to \eqref{eq:correlated}, we get that:
\begin{equation}
\begin{aligned}
& Pr\left( \frac{D}{{C}^{\diamond}_N} - 1 \leq \sum_{i \neq N}\frac{{C}^{\diamond}_i}{{C}^{\diamond}_N}Z_i \leq \frac{D}{{C}^{\diamond}_N}  \right) \\
\leq &\kappa' N^{-\frac{1}{4}}\int_{0}^{D/C^{\diamond_N}}f_{Z_0}(z)dz \\
\leq  & \kappa'N^{-\frac{1}{4}}.
\end{aligned}
\end{equation}

If $\bar{Z} = 0$, then $Z_i$'s are i.i.d. random variables, then $Z_0 = 0$ and the bound naturally carries over, as stated in the following corollary.
\begin{corollary}
	If $Z_i$'s are i.i.d. random variables, then it is a special case of correlated $Z_i$'s in Assumption A3 when $\bar{Z}$, and the upper bound in \eqref{eq:upperbound} is valid for i.i.d. $Z_i$'s, i.e.:
	\begin{equation}
	Pr_{\cdot|D}\left( \frac{D}{{C}^{\diamond}_N} - 1 \leq \sum_{i \neq N}\frac{{C}^{\diamond}_i}{{C}^{\diamond}_N}Z_i \leq \frac{D}{{C}^{\diamond}_N}  \right) \leq \kappa' N^{-\frac{1}{4}}.
	\end{equation}
\end{corollary}




Now we can complete the proof. Going back to \eqref{eqn_derivativeprofit2}, and substituting the above upper bound to get that:
\begin{align}
\pi_N({C}^{\diamond}_1,  \ldots, {C}^{\diamond}_N)  & \leq \gamma \frac{{C}^{\diamond}_N G_N}{G_1} + {C}^{\diamond}_1 \kappa' N^{-\frac{1}{4}}  \nonumber\\
& \leq \gamma \frac{{C}^{\diamond}_N G_N}{G_1} + \frac{\E D}{\gamma N} \kappa' N^{-\frac{1}{4}}.
\end{align}

Next, we upper bound the aggregate payment made to the producers using 
Proposition \ref{prop:expected_profit}. Recall that $G_i \leq G_1$ for all $i$ and that $G_N \geq \gamma \E[Z_N]$  as per Lemma~\ref{lem_expectationgames}. We now get that:
\begin{align*}
& \sum_{i=1}^N \pi_i({C}^{\diamond}_1,  \ldots, {C}^{\diamond}_N)  \\
& \leq \sum_{i=1}^N \pi_N({C}^{\diamond}_1,  \ldots, {C}^{\diamond}_N) \frac{\bar{C_i} G_i}{{C}^{\diamond}_N G_N}\\
& \leq \sum_{i=1}^N \left(\gamma \frac{{C}^{\diamond}_N G_N}{G_1} + \frac{D G_i}{\gamma N G_N} \kappa' N^{-\frac{1}{4}} \right) \frac{\bar{C_i}G_i }{{C}^{\diamond}_NG_N}\\
& \leq \gamma \sum_{i=1}^N {C}^{\diamond}_i \frac{G_i}{G_1} + \frac{\E D}{q\gamma^2 N} \kappa' N^{-\frac{1}{4}} \sum_{i=1}^N \frac{\bar{C_i}}{{C}^{\diamond}_N} \\
& \leq \gamma \sum_{i=1}^N {C}^{\diamond}_i + \frac{\E D}{q\gamma^2 N} \kappa' N^{-\frac{1}{4}} \times N \\
& = \gamma \sum_{i=1}^N {C}^{\diamond}_i +  \frac{\E D}{q\gamma^2} \kappa' N^{-\frac{1}{4}}.
\end{align*}

In the penultimate equation, we used the fact that ${C}^{\diamond}_i \leq {C}^{\diamond}_N$ and so, trivially, $\sum_{i=1}^N \frac{{C}^{\diamond}_i}{{C}^{\diamond}_N} \leq N$. This completes the proof of the second case, and hence, the theorem. \hfill \qed

\section{Proof for Asymmetric Gamma}
Before stating the main result in this section, we extend Lemma~\ref{lem_capupperbound} to arbitrary investment costs to obtain an upper bound on $C^{\diamond}_1$.

\begin{lemma}
	\label{lem_capupperbound_asy}
	Given any equilibrium solution $(C^{\diamond}_1, C^{\diamond}_2, \ldots, C^{\diamond}_N)$ of the two-level game with asymmetric costs $(\gamma_1, \gamma_2 \ldots, \gamma_N)$, it must be the case that $C^{\diamond}_1 \leq \frac{\E D}{\gamma_{min} N}$.
\end{lemma}
\begin{proof}
	Assume by contradiction that the above statement is not true and there exists an equilibrium solution $(C^{\diamond}_1, C^{\diamond}_2, \ldots, C^{\diamond}_N)$ such that $\frac{\E D}{\gamma_{min} N} < C^{\diamond}_1 \leq \ldots \leq C^{\diamond}_N$. Proceeding similarly to the proof of Lemma~\ref{lem_capupperbound}, we can differentiate the profit made by the producer with the highest capacity ($C^{\diamond}_N)$ to get the following:
	\begin{align}
\left(\frac{d}{dC}\pi_N(C^{\diamond}_1, C^{\diamond}_2, \ldots, C)\right)_{C = C^{\diamond}_N} & = \gamma_N \geq \gamma_{\min} \notag \\
\implies \E \left[\mathbbm{1}\left\{(D-\sum_{i \neq N}C^{\diamond}_iZ_i)^+\geq C^{\diamond}_N Z_N\right\} Z_N \right] & \geq \gamma_{min}. \label{eqn_condition_oldinstance}
\end{align}



	Next, consider a `new instance' of the capacity-price game with $N$ providers, all of whom have a symmetric investment cost equal to $\gamma_{\min}$. From Theorem~\ref{th:social_nash}, we know that instances with symmetric investment costs admit symmetric equilibrium solutions. Let $(C', \ldots, C')$ denote the socially optimal investment strategy for this new instance with symmetric investment cost $\gamma_{\min}$, which is also an equilibrium solution. Clearly, for this symmetric equilibrium, the derivative of profit must also equal the investment cost. Therefore:

	\begin{align}
	\left(\frac{d}{dC}\pi_N(C', C', \ldots, C)\right)_{C = C'} & = \gamma_{\min} \notag\\
	\implies \E \left[\mathbbm{1}\left\{(D-\sum_{i \neq N}C'Z_i)^+\geq C' Z_N\right\} Z_N \right] & = \gamma_{\min} \label{eqn_condition_newinstance}
	\end{align}

	Further, recall that in the symmetric equilibrium for the instance where all investment costs are $\gamma_{\min}$, it must be that $C' \leq \frac{\E D}{\gamma_{\min} N}$ as derived in Equation~\eqref{eqn_symmetricupper}. Since $C^{\diamond}_1 > \frac{\E D}{\gamma_{min} N}$ in our original instance, this implies that $C^{\diamond}_1 > C'$. Finally, let $\mathcal{E}$ denote the set of events\footnote{For our purposes, an event is a tuple of instantiations of the i.i.d random variables $(Z_1, Z_2, \ldots, Z_N)$ and $D$ satisfying the required condition.} $\mathbbm{1}\left\{(D-\sum_{i \neq N}C^{\diamond}_iZ_i)^+\geq C^{\diamond}_N Z_N\right\}$ and let $\mathcal{E}'$ denote the events satisfying $\mathbbm{1}\left\{(D-\sum_{i \neq N}C'Z_i)^+\geq C' Z_N\right\}$. Since $C' < C^{\diamond}_1 \leq C^{\diamond}_2 \leq \ldots \leq C^{\diamond}_N$, it is not hard to deduce that $\mathcal{E} \subset \mathcal{E}'$. Indeed for  any (non-zero) instantiation $(Z_1, \ldots, Z_N)$, we have that $(D-\sum_{i \neq N}C^{\diamond}_iZ_i)^+ < (D-\sum_{i \neq N}C'Z_i)^+$ and $C^{\diamond}_NZ_N > C'Z_N$. Combining~\eqref{eqn_condition_oldinstance} and~\eqref{eqn_condition_newinstance} along with the fact that $\mathbbm{1}\left\{(D-\sum_{i \neq N}C^{\diamond}_iZ_i)^+\geq C^{\diamond}_N Z_N\right\} < \mathbbm{1}\left\{(D-\sum_{i \neq N}C'Z_i)^+\geq C' Z_N\right\}$, we get that:
	\begin{align*}
	&\gamma_{\min} \leq \E \left[\mathbbm{1}\left\{(D-\sum_{i \neq N}C^{\diamond}_iZ_i)^+\geq C^{\diamond}_N Z_N\right\} Z_N \right]  < \\ & \quad \E \left[\mathbbm{1}\left\{(D-\sum_{i \neq N}C'Z_i)^+\geq C' Z_N\right\} Z_N \right] = \gamma_{\min} ,
	\end{align*}

	which is a contradiction.
\end{proof}

%
%
%
%
%
%

Now we can proceed to the proof of Theorem \ref{theorem:correlation_asymmetric}.

\begin{proof}
The proof is very similar to that of Theorem~\ref{theorem:correlation}, so we only sketch the new arguments. As with the previous proof, suppose that $G_i = \E[\min(Z_i, D/{C}^{\diamond}_i)]$ for all $1 \leq i \leq N$. Since {$N > \frac{1}{\frac{D_{min}}{D_{max}}\gamma_{\min}}$}, we can invoke Lemma~\ref{lem_capupperbound_asy} to get that $C^{\diamond}_1 \leq \E D$.

Consider the payment received by producer $1$ whose capacity happens to be ${C}^{\diamond}_1$. As per Equation~\eqref{eqn_paymentasymmetric}, this equals:
\begin{align*}
\pi_1({C}^{\diamond}_1, \ldots,   {C}^{\diamond}_N) & = \pi_N({C}^{\diamond}_1,  \ldots, {C}^{\diamond}_N) \frac{{C}^{\diamond}_1 G_1 }{{C}^{\diamond}_N G_N }
\end{align*}
The total profit made by this producer is $\pi_1({C}^{\diamond}_1, \ldots, {C}^{\diamond}_N) - \gamma_1 {C}^{\diamond}_1$. Since this is an equilibrium solution, we have that\\ $\left(\frac{d}{dC}\pi_i(C,  \ldots, {C}^{\diamond}_N)\right)_{C={C}^{\diamond}_1} = \gamma_1$. Proceeding similarly to the proof of Theorem~\ref{th:symmetric_correlation} bearing in mind that $G_1$ is independent of $C$, we get that:
\begin{equation}
	\begin{aligned}
	& \pi_N({C}^{\diamond}_1,  \ldots, {C}^{\diamond}_N) \\
	& = \gamma_1 \frac{{C}^{\diamond}_N G_N}{G_1} \\
	& + {C}^{\diamond}_1 \E[ \mathbbm{1} \{ 0 \leq D-\sum_{i \neq N}{C}^{\diamond}_iZ_i \leq {C}^{\diamond}_NZ_N  \}  Z_1]  \label{eqn_derivativeprofit2_asy}.
	\end{aligned}
\end{equation}

Fix some constant $\kappa$. The rest of the proof proceeds in two cases:

(Case I: $\sum_{i=1}^{N-1} \frac{{C}^{\diamond}_i}{{C}^{\diamond}_N} \leq \kappa N^{\frac{3}{4}}$)

Since, $Z_1 \leq 1$, we can manipulate~\eqref{eqn_derivativeprofit2_asy} to get:
\begin{align*}
\pi_N({C}^{\diamond}_1,  \ldots, {C}^{\diamond}_N) & \leq \gamma_1 \frac{{C}^{\diamond}_N G_N}{G_1} \\
& + {C}^{\diamond}_1 Pr\left(  0 \leq D-\sum_{i \neq N}{C}^{\diamond}_iZ_i \leq {C}^{\diamond}_NZ_N  \right) \\
& \leq \gamma_1 \frac{{C}^{\diamond}_N G_N}{G_1} + {C}^{\diamond}_1 .
\end{align*}
Recall from Lemma~\ref{lem_capupperbound_asy} that in any equilibrium solution, we must have that ${C}^{\diamond}_1 \leq \frac{D}{\gamma_{\min} N}$. Substituting this above, we get that
$$\pi_N({C}^{\diamond}_1,  \ldots, {C}^{\diamond}_N)  \leq \gamma_1 \frac{{C}^{\diamond}_N G_N}{G_1}+ \frac{\E D}{\gamma_{\min} N}.$$

Next, observe that for any $i \neq N$, we can apply 
Proposition \ref{prop:expected_profit} to obtain an upper bound on its profit, namely that:
\begin{align}
\pi_i({C}^{\diamond}_1,  \ldots, {C}^{\diamond}_N)  & \leq \pi_N({C}^{\diamond}_1,  \ldots, {C}^{\diamond}_N) \frac{{C}^{\diamond}_i G_i}{{C}^{\diamond}_N G_N} \nonumber \\
& \leq \left(\gamma_1 \frac{{C}^{\diamond}_N G_N}{G_1} + \frac{\E D}{\gamma_{\min} N} \right) \frac{{C}^{\diamond}_i G_i}{{C}^{\diamond}_N G_N}  \nonumber\\
& \leq \gamma_1 {C}^{\diamond}_i  + \frac{\E D {C}^{\diamond}_i \E[Z_i]  }{ \gamma_{\min} N {C}^{\diamond}_N \gamma_N \E[Z_N]} \notag\\ 
& \leq \gamma_1 {C}^{\diamond}_i  + \frac{\E D {C}^{\diamond}_i }{ \gamma_{\min}^2 N{C}^{\diamond}_N }  \label{eqn_profitupperboundall_asy}.
\end{align}
Recall that $\gamma_N \geq \gamma_{\min}$.

Summing up Equation~\ref{eqn_profitupperboundall_asy} over all $i$ including $i=N$, we get that
$$\sum_{i=1}^N \pi_i({C}^{\diamond}_1,  \ldots, {C}^{\diamond}_N) \leq \gamma_1 \sum_{i=1}^N {C}^{\diamond}_i + \frac{\E D}{N\gamma_{\min}^2} \left(\sum_{i=1}^{N}\frac{{C}^{\diamond}_i}{{C}^{\diamond}_N}\right). $$

Of course, as per our assumption, we have that $\sum_{i=1}^{N-1}\frac{{C}^{\diamond}_i}{{C}^{\diamond}_N} \leq \kappa N^{\frac{3}{4}}.$ Substituting this above, we get that
\begin{align*}
\sum_{i=1}^N \pi_i({C}^{\diamond}_1,  \ldots, {C}^{\diamond}_N) & \leq \gamma_1 \sum_{i=1}^N {C}^{\diamond}_i + \frac{\E D}{\gamma_{\min}^2} \left(\kappa N^{-\frac{1}{4}} + \frac{1}{N}\right)\\
& \leq \gamma_{\max} \sum_{i=1}^N {C}^{\diamond}_i + \frac{\E D}{\gamma_{\min}^2} \left(\kappa N^{-\frac{1}{4}} + \frac{1}{N}\right).
\end{align*}
This proves the theorem statement for the case where $\sum_{i=1}^{N-1} \frac{{C}^{\diamond}_i}{{C}^{\diamond}_N} \leq \kappa N^{\frac{3}{4}}$. 

(Case II: $\sum_{i=1}^{N-1} \frac{{C}^{\diamond}_i}{{C}^{\diamond}_N} > \kappa N^{\frac{3}{4}}$)

Let us go back to Equation~\eqref{eqn_derivativeprofit2_asy} and consider the term \\$B = \E[ \mathbbm{1} \left\{ 0 \leq D-\sum_{i \neq N}{C}^{\diamond}_iZ_i \leq {C}^{\diamond}_NZ_N  \right\}  Z_1]$. By proceeding identically as we did in the proof of Theorem~\ref{theorem:correlation} and applying the Cauchy-Scwarz inequality, we obtain the following bound:
	\begin{equation}
	B \leq \Pr\left( \frac{D}{{C}^{\diamond}_N} - 1 \leq \sum_{i \neq N}\frac{{C}^{\diamond}_i}{{C}^{\diamond}_N}Z_i \leq \frac{D}{{C}^{\diamond}_N}  \right) \leq \kappa' N^{-\frac{1}{4}},.
	\end{equation}




where $\kappa'$ is some constant. Now we can complete the proof. Going back to \eqref{eqn_derivativeprofit2_asy}, and substituting the above upper bound to get that:
\begin{align}
\pi_N({C}^{\diamond}_1,  \ldots, {C}^{\diamond}_N)  & \leq \gamma_1 \frac{{C}^{\diamond}_N G_N}{G_1} + {C}^{\diamond}_1 \kappa' N^{-\frac{1}{4}}  \nonumber\\
& \leq \gamma_1 \frac{{C}^{\diamond}_N G_N}{G_1} + \frac{\E D}{\gamma N} \kappa' N^{-\frac{1}{4}} \label{eqn_derivativeprofit3}.
\end{align}

Next, we upper bound the aggregate payment made to the producers using 
Proposition \ref{prop:expected_profit}. Recall that $G_i \leq G_1$ for all $i$ and that $G_N \geq \gamma_{\min} \E[Z_N]$  as per Lemma~\ref{lem_expectationgames}. We now get that:
\begin{align*}
& \sum_{i=1}^N \pi_i({C}^{\diamond}_1,  \ldots, {C}^{\diamond}_N) \\
 & \leq \sum_{i=1}^N \pi_N({C}^{\diamond}_1,  \ldots, {C}^{\diamond}_N) \frac{\bar{C_i} G_i}{{C}^{\diamond}_N G_N}\\
& \leq \sum_{i=1}^N \left(\gamma_1 \frac{{C}^{\diamond}_N G_N}{G_1} + \frac{\E D G_i}{\gamma_{\min} N G_N} \kappa' N^{-\frac{1}{4}} \right) \frac{\bar{C_i}G_i }{{C}^{\diamond}_NG_N}\\
& \leq \gamma_1 \sum_{i=1}^N {C}^{\diamond}_i \frac{G_i}{G_1} + \frac{\E D}{q\gamma_{\min}^2 N} \kappa' N^{-\frac{1}{4}} \sum_{i=1}^N \frac{\bar{C_i}}{{C}^{\diamond}_N} \\
& \leq \gamma_1 \sum_{i=1}^N {C}^{\diamond}_i +  \frac{\E D}{\gamma_{\min}^2} \kappa' N^{-\frac{1}{4}}.\\
& \leq \gamma_{\max} \sum_{i=1}^N {C}^{\diamond}_i +  \frac{\E D}{q\gamma_{\min}^2} \kappa' N^{-\frac{1}{4}}.
\end{align*}

In the penultimate equation, we used the fact that ${C}^{\diamond}_i \leq {C}^{\diamond}_N$ and so, trivially, $\sum_{i=1}^N \frac{{C}^{\diamond}_i}{{C}^{\diamond}_N} \leq N$. This completes the proof of the second case, and hence, the theorem. \end{proof}

\section{Proof for Theorem \ref{theorem:unique}}
To prove that there is only one symmetric Nash equilibrium, we need to show that there is a unique $C$ such that $C_1 = C_2 = \dots = C_N = C$ which minimizes the total profit in the game:
\begin{equation} \label{eqn:H}
N \E \left[\min \left( (D - C \sum_{j\neq i}^{N}Z_j)^+, CZ_j  \right) \right] - \gamma C N.
\end{equation}

The optimality condition on $C$ to minimize this payment is shown in \eqref{eqn:gamma}.

Showing a unique $C$ maximizes \eqref{eqn:H} is equivalent to showing that the right hand side that involves $C$ in \eqref{eqn:gamma} has only one intersection with $\gamma$, i.e.,  $\E \left[ \mathbbm{1}\left\{(D-C\sum_{j \neq i}^NZ_j)^+\geq C Z_i\right\} Z_i \right]$ is monotonic with respect to $C$. This can be seen from the fact that as $C$ increases, $(D-C\sum_{j \neq i}^NZ_j)^+$ decreases for each realization of $Z_j$ and $CZ_i$ increases by each realization of $Z_i$. Therefore the term inside the  expectation is monotonically decreasing as $C$ increases.
\hfill \qed

{\section{Proof for Theorem \ref{theorem_exist}}}\label{proof_theorem_asymm_exist}

\begin{proof}
We use a standard equilibrium existence result for games with continuous strategy spaces originally proposed by Debreu~\cite{debreu1952social}. We refer the reader to~\cite{dasgupta2015debreu} for a more accessible restatement of the previous result which we will utilize in this proof. In particular, we will prove a slightly stronger claim than the theorem statement, namely that there always exists an equilibrium $(C^{\diamond}_1, C^{\diamond}_2, \ldots, C^{\diamond}_N)$ such that $C^{\diamond}_N = \max\{C^{\diamond}_1, C^{\diamond}_2, \ldots, C^{\diamond}_N\}$---i.e., the maximum capacity investment at equilibrium is by the player with the smallest investment cost $\gamma_N$.

Let us begin by considering a slightly restricted version of our original PV game, where we only allow for investment strategies $\bm{C} = (C_1, C_2, \ldots, C_N)$ such that $C_N \geq C_1, \ldots, C_{N-1}$---i.e., the feasible strategies of this game only include those where PVs $\{1,2,\ldots, N-1\}$ invest capacities smaller than or equal to those by PV $N$. Recall that for all $i \leq N$, PV $i$'s profit is given by: $\pi_i(\bm{C}) - \gamma_i C_i$, where:
\begin{equation}
\pi_i(\bm{C}) = \E\big[\min\{(D-\sum_{j \neq i} Z_{j}C_{j})^{+}, Z_NC_N \} \big] \frac{\E[\min(C_iZ_i, D)]}{\E[\min(C_NZ_N, D)]}.
\label{eqn_utilities_players}
\end{equation}

Note that the expectations are over the randomness in both the demand $D$ and the production $(Z_i)_{i=1}^N$. As per the statement in~\cite{dasgupta2015debreu}, this restricted version of our game admits a Nash equilibrium as long as the following conditions are met:$(i)$ The strategy space available to each player is a compact and convex subset of the Euclidean space and is also upper and lower hemicontinuous and non-empty valued; $(ii)$ $\pi_i(\bm{C}) - \gamma_i C_i$ is continuous in $\bm{C}$ and quasi-concave in $C_i$ for all $i$. It is important to note that the result by~\cite{dasgupta2015debreu} allows for the strategy of PV $i$ to be a function of $\bm{C}_{-i}$ as is the case in our restricted game.

For each player $i < N$, the strategy space available to this player $i$ is $C \in [0,C_N]$. For player $N$ its strategy space is $[\max_{i=1}^{N-1}C_i, C^{\max}]$, where $C^{\max}$ is some finite upper bound on the maximum possible investment by a player. Clearly, each player's strategy space is convex, compact, continuous, and non-empty valued. It only remains for us to prove that the profit function is continuous and quasi-concave. From~\eqref{eqn_utilities_players}, it is not hard to observe that all the three components that make up $\pi_i(\bm{C})$ are continuous in all its inputs, i.e.,
\begin{enumerate}
	\item $\E\big[\min\{(D-\sum_{j \neq i} Z_{j}C_{j})^{+}, Z_NC_N \} \big]$
	\item $\E[\min(C_iZ_i, D)]$
	\item $\E[\min(C_NZ_N, D)]$
\end{enumerate}
are all continuous in the vector of capacities $\bm{C}$. Therefore, it follows that $\pi_i(\bm{C}) - \gamma_iC_i$ is continuous in $\bm{C}$ for all $i$.

To prove quasi-concavity, we need to show that
\begin{align*}
\pi_i(\lambda C^{(1)}_i + (1-\lambda)C^{(2)}_i, \bm{C}_{-i}) - \gamma_i( \lambda C^{(1)}_i + (1-\lambda)C^{(2)}_i) \\ \quad \geq \min\{\pi_i(C^{(1)}_i, \bm{C}_{-i}) - \gamma_i C^{(1)}_i, \pi_i(C^{(2)}_i, \bm{C}_{-i})- \gamma_iC^{(2)}_i\},
\end{align*}

for arbitrary $C^{(1)}_i, C^{(2)}_i \leq C_N$ and all $\lambda \in [0,1]$. Without loss of generality, fix some $C^{(1)}_i \leq C^{(2)}_i \leq C_N$ and $\lambda \in [0,1]$. In order to prove quasi-concavity via the above inequality, it is sufficient to prove that the derivative of the PV profit, i.e., $\frac{d}{dC_i}(\pi_i(C_i, \bm{C}_{-i}) - \gamma_iC_i)$ is monotonically decreasing in $C_i$. Indeed, suppose that $\tilde{C}_i$ denotes the point where
$$\frac{d}{dC_i}(\pi_i(C_i, \bm{C}_{-i}) - \gamma_iC_i)_{C_i = \tilde{C}_i} = 0.$$

Then, the derivative being monotonically non-increasing implies that the function $\pi_i(C_i, \bm{C}_{-i}) - \gamma_iC_i$ is increasing from $C_i=0$ to $C_i=\tilde{C}_i$ and decreasing from $C_i=\tilde{C}_i$ to $C_i = C_N$. This naturally implies quasi-concavity. In order to prove the required condition on the derivative, consider differentiating $\pi_i(\bm{C})$ with respect to $C_i$. This gives us:
\begin{align*}
\frac{d}{dC_i}\pi_i(C_i, \bm{C}_{-i}) =  \frac{1}{{C}_NG_N} \Big(\pi_N(\bm{C}) \frac{d}{dC_i}\E[\min(C_iZ_i, D)]   \\
 \quad  - \E[\min(C_iZ_i, D)]) \E[  \mathbbm{1} \{ 0 \leq D-\sum_{i \neq N}{C}_iZ_i \leq {C}_NZ_N  \}  Z_i]\Big),
\end{align*}
where $G_N = E[\min\{Z_N, \frac{D}{C_N}\}]$. In the above differentiation, we used~\eqref{eqn_utilities_players} to obtain the expression for $\pi_i(\bm{C})$ and Proposition~\ref{lem_derivation} for the expression $\frac{d}{dC_i}\pi_N(\bm{C})$. The fact that $\frac{d}{dC_i}(\pi_i(C_i, \bm{C}_{-i}) - \gamma_iC_i)$ is non-increasing in $C_i$ comes from the observations below:
\begin{enumerate}
	\item $\pi_N(\bm{C}) = E[\min((D-\sum_{j\neq N}C_iZ_i)^+, C_NZ_N)]$ is non-increasing in $C_i$.
	\item $\frac{d}{dC_i}\E[\min(C_iZ_i, D)] = E[\mathbbm{1}\left\{D \geq C_iZ_i\right\} Z_i ]$ is non-increasing in $C_i$.
	\item $\E[\min(C_iZ_i, D)])$ is non-decreasing in $C_i$ and therefore, $-\E[\min(C_iZ_i, D)])$ is non-increasing in the same argument.
	\item $\E[  \mathbbm{1} \left\{ 0 \leq D-\sum_{i \neq N}{C}_iZ_i \leq {C}_NZ_N  \right\}  Z_i$ is non-decreasing in $C_i$ and thus, its negative is non-increasing in the same argument.
\end{enumerate}

To sum up, we have shown that $\frac{d}{dC_i}\pi_i(\bm{C})$ and therefore, $\frac{d}{dC_i}(\pi_i(\bm{C}) - \gamma_iC_i)$ is non-increasing in $C_i$, which in turn implies quasi-concavity. We can then apply the equilibrium existence result from~\cite{debreu1952social,dasgupta2015debreu} to show that the restricted version of our original PV game where player $N$ has the highest capacity investment always admits a Nash equilibrium. To complete the proof, we need to show that the Nash equilibrium for this restricted game is also an equilibrium for our original PV game with no restrictions on the player strategies. Suppose that $\bm{C}^{\diamond} = (C^{\diamond}_1, C^{\diamond}_2, \ldots, C^{\diamond}_N)$ denotes an equilibrium for the restricted game and assume by contradiction that this is not an equilibrium for the original game. Then, only one of two possibilities can occur: $(i)$ either there exists player $i < N$, who can strictly improve its profit by increasing its capacity investment to $C'_i > C^{\diamond}_N$ or $(ii)$ player $N$ can strictly increase its profit by lowering its capacity investment to $C'_N < \max_{i \neq N}C^{\diamond}_i$.

Consider the first case. We will prove that if there exists such a player $i$, then it must also be true that player $N$ can increase its capacity investment and consequentially, its profit which contradicts the fact that $\bm{C}^{\diamond}$ is an equilibrium for the restricted game. If such an improving strategy exists for player $i$, it must definitely be the case that the right hand derivative of its profit at $C_i = C^{\diamond}_N$ must be strictly positive---this is because the derivative of player $i$'s profit is monotonically non-increasing in $C_i$ when $C_i \geq C^{\diamond}_N$. That is, we have that:
\begin{align*}
\frac{d}{dC_i}(\pi_i(C_i, \bm{C}^{\diamond}_{-i}))_{C_i = (C^{\diamond}_N)^+}  = \\ ~~ \frac{d}{dC_i}(E[\min((D-\sum_{j\neq i}C^{\diamond}_jZ_j)^+, C_iZ_i)])_{C_i = C^{\diamond}_N} > \gamma_i.
\end{align*}

Note that the derivative in the right hand side above equals $$-\E \left[\mathbbm{1}\{0 \leq (D - \sum_{j \neq i}C^{\diamond}_jZ_j) \leq C^{\diamond}_N Z_i\}Z_i\right]$$ by Proposition~\ref{lem_derivation}. However, since $\gamma_i \geq \gamma_N$, this in turn implies that
\begin{align*}
\frac{d}{dC_N}(\pi_N(C_N, \bm{C}^{\diamond}_{-N}))_{C_N = (C^{\diamond}_N)}  = \\ ~~ -\E \left[\mathbbm{1}\{0 \leq (D - \sum_{j \neq N}C^{\diamond}_jZ_j) \leq C^{\diamond}_N Z_N\}Z_N\right] > \gamma_N.
\end{align*}

Recall by our assumption that $Z_N$ and $Z_i$ are identically distributed conditional on $(Z_j)_{j \neq i, N}$. In words, this means that player $N$ can also increase its capacity and its profit. This is of course a contradiction of the fact that $\bm{C}^{\diamond}$ is an equilibrium for the restricted game.

In the second case, suppose that player $N$ can lower its capacity and improve profits. The proof in this case is quite analogous to the previous case. Indeed, if player $N$ can lower its capacity to $C'_N < \max_{j \neq N}C^{\diamond}_j$ and increase profits, then one can use monotonicity arguments similar to before to show that the player $i = \arg\max_{j \neq N}C^{\diamond}_j$ can also lower its capacity investment and improve profits. Once again, a contradiction of the fact that $\bm{C}^{\diamond}$ is an equilibrium for the restricted game. This completes our proof that there exists a Nash equilibrium $\bm{C}^{\diamond}$ for the PV game with asymmetric investment costs.


\end{proof}

\end{document}